\documentclass{scrartcl}

\usepackage{graphicx}
\usepackage[utf8]{inputenc}
\usepackage[english]{babel}

\usepackage{amsmath,amssymb,enumerate,url,appendix,tabularx,stmaryrd,mathrsfs}

\usepackage{comment,afterpage}
\usepackage{subcaption}
\usepackage{framed}
\usepackage{mathtools}
\usepackage{fullpage}

\usepackage{amsthm}
\usepackage{graphicx}

\theoremstyle{plain}
\newtheorem{theorem}{Theorem}[section]
\newtheorem{lemma}[theorem]{Lemma}
\newtheorem{corollary}[theorem]{Corollary}
\newtheorem{proposition}[theorem]{Proposition}

\newtheorem{assumption}{Assumption}

\newtheorem{definition}[theorem]{Definition}

\newtheorem{remark}[theorem]{Remark}

\numberwithin{equation}{section}
\numberwithin{figure}{section}
\numberwithin{table}{section}

\newcommand{\R}{\mathbb{R}}
\newcommand{\C}{\mathbb{C}}
\newcommand{\N}{\mathbb{N}}

\newcommand{\bigO}{\mathcal{O}}

\newcommand{\abs}[1]{\left|#1\right|}
\newcommand{\norm}[1]{\left\|#1\right\|}

\newcommand{\laplace}{\Delta}

\def\ii{i}

\newcommand{\XX}{\mathcal X}
\newcommand{\YY}{\mathcal Y}

\newcommand{\tracejump}[1]{\llbracket \gamma #1 \rrbracket}
\newcommand{\normaltracejump}[1]{\llbracket \gamma_{\nu} #1\rrbracket}

\newcommand{\normalDjump}[1]{\llbracket \partial_{\nu} #1\rrbracket}

\newcommand{\bs}[1]{\boldsymbol #1}

\newcommand{\ones}{\mathbf{1}}

\newcommand{\dom}{\operatorname{dom}}
\DeclarePairedDelimiter\floor{\lfloor}{\rfloor}

\def\rkA{\mathcal{Q}}
\def\rkb{\mathbf{b}}
\def\rkc{\mathbf{c}}
\def\lifting{\mathscr{E}}

\def\dd{\partial^k}

\newcommand{\Cpspace}[2][\mathcal{X}_{\mu}]{\mathcal{C}^{#2}\left(\left[0,T\right],#1\right)}

\newcommand{\ztrafo}[1]{\mathscr{Z}\left[#1\right]}
\newcommand{\ltrafo}{\mathscr{L}}

\def\AA{A}
\def\AAstar{\AA_{\star}}

\def\id{I}

\renewcommand{\Re}{\operatorname{Re}}

\iffalse %%change this to \iftrue to rebuild the figures from the tex files using Tikz/pgfplots
\usepackage{tikz}
\usetikzlibrary{shapes}
\usepackage{pgfplots}
\newcommand{\includeTikzOrEps}[1]{\tikzexternalenable \tikzsetnextfilename{#1}  {\include{figures/#1}} \tikzexternaldisable}

\usetikzlibrary{external}
\tikzexternaldisable

\else
\newcommand{\includeTikzOrEps}[1]{\includegraphics{figures_pdf/#1}}
\fi

\title{ Runge-Kutta approximation for $C_0$-semigroups in the graph norm with applications to time domain boundary integral equations  }
\author{Alexander Rieder, Francisco-Javier Sayas, Jens Markus Melenk }

\date{\today}

\begin{document}

\maketitle
\abstract{We consider the approximation to an abstract evolution problem with inhomogeneous side constraint using $A$-stable Runge-Kutta methods. 
  We derive a priori estimates in norms other than the underlying Banach space.
  Most notably, we derive estimates
    in the graph norm of the generator.
  These results are used to study
convolution quadrature based discretizations of a wave scattering and a heat conduction problem.}

\section{Introduction}
Many time dependent partial differential equations can be conveniently described in the language of strongly continuous semigroups.
In this language, these initial boundary value problems resemble systems of ordinary differential equations, which suggests that they are amendable
to the standard discretization schemes of multistep or Runge-Kutta type. Unlike the ODE case, one needs to pay special attention to the
boundary conditions imposed by the generator of the semigroup. This, in most cases, leads to a reduction of order phenomenon, meaning
that the convergence rates are (mainly) determined by the stage order of the Runge-Kutta method instead of the classical order.
The a priori convergence of Runge-Kutta methods for semigroups has been extensively studied in the literature. Starting with the
early works~\cite{brenner_thomee_rat_approx_semigroup,crouzeix_rk_approx_evolution}, it has been established that
conditions of the form $u(t) \in \dom(\AA^{\mu})$, where $\AA$ is the generator of the semigroup, determine the convergence rates.
In~\cite{ostermann_roche}, this has been generalized to the case of non-integer $\mu\geq 1$ using the theory of interpolation spaces.
Finally, in~\cite{mallo_palencia_optimal_orders_rk}, the case of $\mu \in [0,1]$ was adressed, which is
the case needed for PDEs with  inhomogeneous boundary conditions.
All of these works focus on establishing convergence rates with respect to the norm of the underlying Banach space.
In many applications one needs to establish convergence with respect to other norms, for example, in order to
be able to bound boundary traces of the solution. Most notably, one might be interested in convergence of $\AAstar u$,
where $\AAstar$ is an extension of the generator that disregards boundary conditions. 
If $u$ is assumed to be in $\dom(\AA)$, we get $\AAstar u=\AA u$ and the convergence result can be easily established by using the fact that
the time-evolution commutes with the generator of the underlying semigroup (both in the continuous and discrete settings).
If the boundary conditions are inhomogeneous, such a strategy cannot be pursued. It is the goal of this
paper to establish convergence results for $\AAstar u$ also for the case $u(t) \in \dom(A^{\mu})$ for $\mu \in [0,1]$,
again using the theory of interpolation spaces.

Similarly it is sometimes useful to compute discrete integrals of the time evolution, by reusing the same Runge-Kutta method. Also in this case, we
establish rigorous convergence rates.

Our interest in such estimates originally arose from the study of time domain boundary integral equations
and their discretization using convolution quadrature(CQ). It has already been noticed in the early works (see e.g.~\cite{lubich_ostermann_rk_cq}) that
such discretizations have a strong relation to the Runge-Kutta approximation of the underlying semigroup.
This approach of studying TDBIEs in a strictly time-domain way has recently garnered a lot of interest, see~\cite{sayas_hassel_fembem,banjai_laliena_sayas_kirchhoff_formulas,sayas_new_analysis}
and the monograph~\cite{sayas_book}, as it potentially allows sharper bounds than the more standard Laplace domain based approach. Similar techniques have even been extended to the case of certain nonlinear problems in~\cite{banjai_and_me}.
This paper can be seen as our latest addition to this effort.
While the convergence rates provided by the Laplace-domain approach in~\cite{BanLM} and the results in this current paper are
  essentially the same, the present new approach provides better insight into the dependence on the
  end-time of the computation. It also fits more naturally with the time-domain analysis
of the continuous problem and space discretization, as for example presented in~\cite{sayas_new_analysis}.

The paper is structured as follows. Section~\ref{sect:setting} introduces the abstract setting and fixes notation, most notably for working with Runge-Kutta methods.
Section~\ref{sect:error_estimates} then contains the main estimates. Starting by summarizing known results from~\cite{mallo_palencia_optimal_orders_rk}
in Section~\ref{sect:summary_of_AMP}, we then  formulate the main new results of this article in Section~\ref{sect:new_results}.
After proving some preparatory Lemmas
related to Runge-Kutta methods  
  in Section~\ref{sect:computations} and~\ref{sect:some_lemmas}, we then provide
the proofs of the main estimates in Section~\ref{sect:proofs}.
In Section~\ref{sect:applications}, to showcase how
the theory developed in this paper is useful for this class of problems, we consider a simple exterior scattering problem in Section~\ref{sec:9}
and a heat transmission problem in Section~\ref{sect:heat}.
We note that Section~\ref{sec:9} showcases the need for the bound on the discrete integral
  of the result, whereas Section~\ref{sect:heat} was chosen because, in order to bound
  the main quantity of interest on the boundary, we need to apply a trace theorem. This
necessitates the use of the graph norm estimate.

\section{Problem setting}
\label{sect:setting}

We start by fixing the general setting used for the rest of the paper, first with respect to the equation to be solved and then with respect to its discretization. 

\subsection{Operator equation, functional calculus, and Sobolev towers}

\begin{assumption}
  \label{ass:1.1}
  We are given:
  \begin{itemize}
  	\item[(a)] a closed linear operator $\AAstar:\operatorname{dom}\AAstar\subset  \mathcal X \to \mathcal X$ in
  		 a Banach space $\mathcal{X}$,
  	\item[(b)] and a bounded linear operator $B: \operatorname{dom} \AAstar \to \mathcal{M}$.
  \end{itemize}
We assume that $A:=\AAstar|_{\ker{B}}$ generates a $C_0$-semigroup and
  that $B$ admits a bounded right inverse, denoted by $\lifting$, such that $\mathrm{range}\,\lifting \subset \ker (I-\AAstar)$, where $I:\mathcal X\to \mathcal X$ is the identity operator.
\end{assumption}

We are given $u_0\in \operatorname{dom} \AA$ and data functions $F \in C^1([0,T], \mathcal{X})$, $\Xi \in C^1([0,T],\mathcal{M})$, and we consider the problem: find $u \in C^1([0,T],\mathcal{X})$ such that
\begin{subequations}
\label{eq:1.1}
\begin{alignat}{6}  
  \dot{u}(t) &= \AAstar u(t) + F(t), & \qquad t>0, \\
  B u(t)&=\Xi(t), & \qquad t> 0, \\
  u(0) &=u_0.
\end{alignat}
\end{subequations}
For conditions on the well-posedness of this 
problem, see \cite{sayas_new_analysis}.
We start by recalling the following consequence of the Hille-Yosida theorem.

\begin{proposition}\label{prop:2.1}
  If $A$ is the generator of a $C_0$-semigroup on a Banach space $\mathcal{X}$, then
  there exist constants $\omega \geq 0$ and $M\geq 1 $ such that
  the spectrum $\sigma(A)$ of $A$ satisfies
  $\sigma(A) \subseteq \{z \in \C: \mathrm{Re}\,z \leq \omega\}$ 
  and the resolvent satisfies the estimates
  \begin{align}
    \label{eq:residual_estimate_semigroups}
    \norm{\left(A-z I\right)^{-1}}_{\mathcal{X} \to \mathcal{X}}\leq \frac{M}{ \mathrm{Re}\,z - \omega}
    	\qquad \forall z \mbox{ s.t. } \mathrm{Re}\,z>\omega.
  \end{align}
\end{proposition}
\begin{proof}
The case $\omega=0$ is shown in \cite[Corollary 3.6]{pazy}. The more general case follows as usual by considering the shifted semigroup $e^{\omega t} e^{A t}$ or directly from the integral representation of the resolvent.
\end{proof}

When working with Runge-Kutta methods, 
it is useful to use a calculus that allows us to apply rational functions to (unbounded) operators, as long as the poles of the function
are compatible with the spectrum of the operator.

\begin{definition}[Rational functions of operators]
\label{def:2.3}
Let $q$ be a rational function that is bounded at infinity, $\Lambda$ be the set of poles of $q$, which we can write in the form (note that we allow for some of the factors in the numerator to be constant)
\[
q(z)=c_0 \prod_{i=1}^n \frac{c_i z-1}{z-\lambda_i}
	=c_0\prod_{i=1}^n \left(c_i+\frac{c_i\lambda_i-1}{z-\lambda_i}\right). 
\] 
If $A:\dom A\subset \mathcal X\to \mathcal X$ is a linear operator such that $\sigma(A)\cap \Lambda=\emptyset$, we define
\begin{equation}\label{eq:2.4}
q(A):= c_0 (c_1\id+(c_1\lambda_1-1)(A-\lambda_1 \id)^{-1})
		\ldots
	   (c_n\id+(c_n\lambda_n-1)(A-\lambda_n \id)^{-1}).	
         \end{equation}
       \end{definition}
It is easy to see that different reorderings of the factors in the numerator and denominator of $q$ produce the same result and that each factor in the definition of $q(A)$ is a bounded linear operator in $\mathcal X$ since $\lambda_i\not\in \sigma(A).$ The  bounded linear operator $q(A):\mathcal X \to \mathcal X$ satisfies
\begin{equation}\label{eq:2.3}
  \| q(A)\|_{\mathcal X \to \mathcal X} \le
  C_q \Big(1+ \big( \max_{\lambda\in \Lambda}\| (A-\lambda \id)^{-1}\|_{\mathcal X\to \mathcal X}
\big)^n\Big).
\end{equation}

%\begin{proof}
%We can bring $\mathcal S$  to Jordan form, and write $\mathcal S=\mathcal P \mathcal J \mathcal P^{-1}$ where $\mathcal J$ consists of Jordan-blocks and most notably is an upper triangular matrix with eigenvalues $(\mu_j)_{j=1}^{m}$ of $\mathcal S$ on its diagonal. 
%We can thus write 
%\[
%\mathcal I\otimes A  - \mathcal S\otimes \id =
%(\mathcal P\otimes \id) (\mathcal I\otimes A-\mathcal J\otimes \id) (\mathcal P^{-1}\otimes \id),
%\]
%We now note that $\mathcal P^{-1}\otimes \id=(\mathcal P\otimes \id)^{-1}$ and that the inverse of $\mathcal I\otimes A-\mathcal J\otimes \id$ can be built by backward substitution, each entry involving the inversion of $(A-\mu_j I)$ and linear combinations of the entries computed before. The bound for the norm follows easily from this argument.
%\end{proof}

The error estimates of this paper use the theory of interpolation spaces. For Banach spaces $\mathcal{X}_1 \subset \mathcal{X}_0$ with continuous embedding and $\mu \in (0,1)$,
we define the space $[\mathcal{X}_0,\mathcal{X}_1]_{\mu,\infty}$ using real interpolation with the following norm:
\begin{align}
  \label{eq:definition_interpolation_norm}
  \norm{u}_{[\mathcal{X}_0,\mathcal{X}_1]_{\mu,\infty}}
  &:=\operatorname{ess\,sup}_{t>0}{\left(t^{-\mu} \inf_{v \in \mathcal{X}_1} \left[\norm{u-v}_{\mathcal{X}_0} + t \norm{v}_{\mathcal{X}_1}\right]\right)}.
\end{align}
We will not go into details of the definitions, they can be found in \cite{triebel95,tartar07} or \cite[Appendix B]{mclean}.
For simplicity of notation we often drop the second parameter $\infty$ and just write
$[\mathcal{X}_0,\mathcal{X}_1]_{\mu}$.

The most important property is the following: a bounded linear operator $T: \mathcal{X}_0 \to \mathcal{Y}_0$ and $\mathcal{X}_1 \to \mathcal{Y}_1$
with $\XX_1 \subseteq \XX_0$ and $\YY_1 \subseteq \YY_0$
is also a bounded operator mapping
$[\mathcal{X}_0,\mathcal{X}_1]_{\mu} \to [\mathcal{Y}_0,\mathcal{Y}_1]_{\mu}$ with the following norm bound
\begin{align}
  \label{eq:interpolation_thm_est}
  \norm{T}_{[\mathcal{X}_0,\mathcal{X}_1]_{\mu} \to [\mathcal{Y}_0,\mathcal{Y}_1]_{\mu}}
  \leq   \norm{T}_{\mathcal{X}_0 \to \mathcal{Y}_0}^{1-\mu} \norm{T}_{\mathcal{X}_1 \to \mathcal{Y}_1}^{\mu}.
\end{align}
We also note that for $\mu_1 \leq \mu_2$, the spaces are nested, i.e.,
$[\mathcal{X}_0,\mathcal{X}_1]_{\mu_2} \subseteq [\mathcal{X}_0,\mathcal{X}_1]_{\mu_1}$ with a continuous embedding. For notational convenience
we write $[\mathcal{X}_0,\mathcal{X}_1]_{0}:=\mathcal{X}_0$ and $[\mathcal{X}_0,\mathcal{X}_1]_{1}:=\mathcal{X}_1$. We will be interested in a collection of spaces defined by interpolating the domains of the powers of the operator $A$. The details of this construction can be found, for example in \cite{engel_nagel}.

\begin{definition}[Sobolev towers]
  \label{def:2.4}
  Let $\AA$ be a closed operator on a Banach space $\mathcal{X}$. For $\mu \in \N_0$, we define the following spaces
  $\mathcal{X}_0:=\operatorname{dom} \AA^0:=\mathcal{X}$ and $\mathcal{X}_{\mu}:=\operatorname{dom}\AA^{\mu}$, equipped with the following norm
  $$\norm{u}_{\mathcal{X}_{\mu}}:=\sum_{j=0}^{\mu}{\norm{\AA^j u}_{\mathcal{X}}}.$$
  For $\mu \in [0,\infty)$, we define  $\mathcal{X}_{\mu}:=\left[\mathcal{X}_{\floor{\mu}}, \mathcal{X}_{\floor{\mu}+1}\right]_{\mu-\floor{\mu}}$
  by interpolation.
\end{definition}

\subsection{Runge-Kutta approximation and discrete stage derivative}

An $m$-stage Runge-Kutta method is given by its Butcher tableau,
characterized by $\rkA \in \R^{m\times m}$ and $\rkb$, $\rkc \in \R^{m}$.
The Runge-Kutta approximation of the Problem~\eqref{eq:1.1} starts at $u^k_0:=u_0$ and then computes for $n\ge 0$ the stage vector $U^k_n \in \mathcal X^m$ and the step approximation $u^k_{n+1}\in \mathcal X$ by solving
\begin{subequations}
\label{eq:1.2}
\begin{align}
  U^k_n &= \ones  \, u^k_n+ k (\rkA \otimes \AAstar) U^k_n + k \rkA F(t_n + k \,\rkc),  \label{eq:1.2a}\\
  (\mathcal I\otimes B ) U^k_n&=\Xi\left(t_n + k \,\rkc\right), \label{eq:1.2b}\\
  u^k_{n+1} &= u^k_n + k (\rkb^\top \otimes \AAstar) U^k_n + k \rkb^\top  F(t_n + k \,\rkc). \label{eq:1.2c} 
\end{align}
\end{subequations}
We have used the following notation (the spaces $\mathcal Y$ and $\mathcal Z$ are generic):
\begin{itemize}
	\item[(a)] if $G:[0,T]\to \mathcal Y$, then
		\[
			G(t_n+k\mathbf c):=
				\left( \begin{array}{ccc} G(t_n+k c_1), & \ldots & 
				G(t_n+ k c_m)\end{array}\right)^\top \in \mathcal Y^m,
		\]
	\item[(b)] if $\mathcal S\in \mathbb R^{m\times m}$ and $C:\mathcal Y\to \mathcal Z$,
		\[
			\mathcal S\otimes C:=
				\left[\begin{array}{ccc} 
					\mathcal S_{11} C & \cdots & \mathcal S_{1m} C \\
					\vdots & & \vdots \\
					\mathcal S_{m1} C & \cdots & \mathcal S_{mm} C
				\end{array}\right] : \mathcal Y^m \to \mathcal Z^m,
		\]
	\item[(c)] if $C:\mathcal Y \to \mathcal Z$,
		\[
			\rkb^\top \otimes C :=
				\left[\begin{array}{ccc} b_1 C & \cdots & b_m C \end{array}\right] : \mathcal Y^m \to \mathcal Z,
		\]
	\item[(d)] $\mathcal I$ is the $m\times m$ identity matrix, 
	and $\mathbf 1=(1,\cdots,1)^\top$,
	\item[(e)] we admit shortened expressions such as
		\begin{alignat*}{6}
		\rkA F(t_n+k\rkc)&:=(\rkA\otimes I) F(t_n+k\rkc),\\
		\ones\, u  &:= (\ones\otimes I) u,\\
		\rkb^\top F(t_n+k\rkc)&:=(\rkb^\top\otimes I) F(t_n+k\rkc).
		\end{alignat*}
\end{itemize}

The following lemma involving inversion of matrices of operators associated to an operator
can be proved by taking the Jordan canonical form of the matrix $\mathcal S$.  
\begin{lemma}
\label{lemma:2.2}
If $A:\dom A\subset \mathcal X\to \mathcal X$ is a linear operator on a Banach space $\mathcal{X}$ and $\mathcal S\in \C^{m\times m}$ satisfies $\sigma(A) \cap \sigma(\mathcal S) = \emptyset$, then
\[
   \mathcal I\otimes A - \mathcal S\otimes I : (\dom A)^m \to \mathcal X^m,
\]
is invertible. Furthermore, there exists a constant $C_{\mathcal S}$, depending only on $\mathcal S$, such that
\[
\| (\mathcal I\otimes A - \mathcal S\otimes I)^{-1}\|_{\mathcal X^m\to \mathcal X^m}
	\le C_{\mathcal S} \,\left[1+\max_{\mu \in \sigma(\mathcal S)} \| (A-\mu \,I)^{-1}\|_{\mathcal X\to \mathcal X}\right]^m.
\]
\end{lemma}
      
Under Assumption \ref{ass:1.1}, the internal stage computation in the RK method can be decomposed in the following form:
\begin{subequations}\label{eq:1.3}
\begin{alignat}{6}
Y^k_n &:= (\mathcal I\otimes \mathscr E) \Xi(t_n+k\,\mathbf c),\\
\label{eq:1.3b}
Z^k_n-k(\mathcal Q\otimes A)Z^k_n
		&=\ones u^k_n-Y^k_n+k \rkA (Y^k_n+F(t_n+k\,\mathbf c)),\\
U^k_n &:=Y^k_n+Z^k_n.
\end{alignat}
\end{subequations}
In \eqref{eq:1.3b} we look for $Z^k_n\in (\mathrm{dom}\,A)^m$.

The stability function of the Runge-Kutta method is the rational function
$r(z):=1+z\rkb^\top(I-z\rkA)^{-1}\ones$.  We will not consider the full class of Runge-Kutta methods, but will restrict our considerations to those satisfying the following Assumptions:
\begin{assumption}
\label{ass:1.2}
\begin{enumerate}[(i)]
  \item \label{ass:rkA_is_invertible} The matrix $\rkA$ is invertible
  \item \label{ass:rk_is_a_stable} The stability function $r$ does not have poles in $\{ z\,:\, \mathrm{Re}\,z<0\}$, and $\abs{r(\ii t)}\leq 1$ for all $t \in \R$ (i.e., the method is $A$-stable). Equivalently, $|r(z)|<1$ for all $z$ with negative real part.
\end{enumerate}
\noindent 
\end{assumption}
We note that Assumption~\ref{ass:1.2}~(\ref{ass:rkA_is_invertible})
implies that the following limit exists
  \[
  \lim_{z\to\infty} r(z)=1-\rkb^\top\rkA^{-1}\ones=:r(\infty).
\]
Assumption~\ref{ass:1.2}~(\ref{ass:rk_is_a_stable}) implies that
\[
\sigma(\mathcal Q)\subset \mathbb C_+:=\{ z\in \mathbb C\,:\, \mathrm{Re}\,z>0\},
\]
and that $r$ is a rational function with poles only in $\mathbb C_+$ and bounded at infinity.

The computation of the internal stages in the numerical approximation \eqref{eq:1.2} requires the inversion of
\[
\mathcal I\otimes \id-k (\rkA\otimes A)
=(\rkA\otimes \id)(\rkA^{-1}\otimes \id-\mathcal I\otimes (k\,A)),
\]
as can be seen from the equivalent form \eqref{eq:1.3}.

If $A$ is the infinitesimal generator of a $C_0$-semigroup, $\omega$ and $M$ are given by Proposition \ref{prop:2.1}, and we choose (recall that $\sigma(\rkA)\subset \C_+$)
\begin{equation}\label{eq:2.22}
k_0< \omega^{-1} d_0, \qquad d_0:=\min\{ \mathrm{Re}\,\lambda\,:\, \lambda\in \sigma(\rkA^{-1})\},
\end{equation}
then the RK method can be applied for any $k\le k_0$. By Proposition \ref{prop:2.1} and Lemma \ref{lemma:2.2}, it follows that
\begin{equation}\label{eq:2.23}
\| (\mathcal I\otimes \id-k (\rkA\otimes A))^{-1}\|_{\mathcal X^m \to \mathcal X^m}
\le C_{\rkA} \frac{M}{d_0-k_0\omega}, \qquad \forall k\le k_0.
\end{equation}

Using Definition \ref{def:2.3}, we can define $r(k\, A)$ for an RK method satisfying Assumption \ref{ass:1.2} and $k\le k_0$ satisfying \eqref{eq:2.22}. We then define
\begin{align}
\label{eq:def_rho_k}
  \rho_k(T):=\sup_{0\leq n k\leq T} \norm{r(kA)^{n}}_{\mathcal{X} \to \mathcal{X}}.
\end{align}
This quantity is relevant for the  study of the error propagation in the Runge-Kutta method.

Given an RK method, we consider the following matrix-valued rational function
\begin{equation} 
\label{eq:rk_def_delta}
    \delta(z):= \left( \rkA - \frac{z}{1-z} \ones \rkb^\top\right)^{-1}
    	= \rkA^{-1}-\frac{z}{1-r(\infty)z}\rkA^{-1}\ones\rkb^\top\rkA^{-1}.
\end{equation}
(The verification that these two formulas correspond to the same matrix is
simple by using the Sherman-Morrison-Woodbury formula.) This matrix is related to the discrete differentiation process associated to an RK method satisfying Assumption \ref{ass:1.2}: on the one hand $k^{-1}\delta(z)$ is the discrete symbol associated to the Discrete Operational Calculus built with the RK method \cite{lubich_ostermann_rk_cq}; on the other hand, a direct interpretation of this symbol is possible using $Z$-transforms (see \cite[Section 6]{HS2016}). Given a sequence $U:=\{ U_n\}$ (tagged from $n\ge 0$) on a space, its $Z$-transform is the formal series
\[
\widehat U(z):=\sum_{n=0}^\infty U_n z^n.
\]
For a detailed treatment on formal power series, see~\cite{henrici_complex_analysis}.

\begin{definition}
Let $U:=\{ U_n\}$ and $V:=\{ V_n\}$ be two sequences  in $\XX^m$ and let $\widehat U$ and $\widehat V$ be their respective $Z$-transforms. If
\[
k^{-1} \delta(z)\widehat U(z)=\widehat V(z),
\]
we write
\[
\partial^k U=V, \qquad U=(\partial^k)^{-1}V.
\]
\end{definition}

The above definition is consistent with the RK discrete operational calculus of Lubich and Ostermann,
see Section~\ref{sect:bem_and_cq} and \cite{lubich_ostermann_rk_cq}.
We now show an explicit form of the computation of $\partial^k$ and its inverse.
\begin{lemma}\label{lemma:2.6}
If $U=\{U_n\}$ is a sequence in $\mathcal X^m$, then $X:=(\partial^k)^{-1}U$ can be computed with the recurrence
\begin{equation}\label{eq:2.dder}
x_0:=0, \qquad 
	\begin{array}{l} 
		X_n:=\ones x_n+k\rkA U_n,\\
		x_{n+1}:=x_n+k\rkb^\top U_n
				=r(\infty) x_n+\rkb^\top\rkA^{-1} X_n,
	\end{array}
\end{equation}
and $V:=\partial^k U$ can be computed with the inverse recurrence
\begin{equation}\label{eq:2.dantider}
u_0:=0, \qquad 
	\begin{array}{l} 
		V_n:=k^{-1}\rkA^{-1} (U_n-\ones u_n),\\
		u_{n+1}:=u_n+k\rkb^\top V_n=r(\infty) u_n+\rkb^\top\rkA^{-1}U_n.
	\end{array}
\end{equation}
\end{lemma}

\begin{proof}
The proof of \eqref{eq:2.dder}  is a simple exercise in $Z$-transforms, while \eqref{eq:2.dantider} follows from \eqref{eq:2.dder} by writing $U_n$ in terms of $X_n$ (and changing names to the sequences).\end{proof}  
Note that the first result of Lemma~\ref{lemma:2.6} says that if we apply the RK method to the equation
\[
\dot x(t)=u(t), \quad x(0)=0, 
	\qquad \mbox{i.e.}, \qquad x(t)=\int_0^t u(\tau)\mathrm d\tau,
\]
and $X:=\{ X_n\}$ is the sequence of vectors of internal stages, then $X=(\partial^k)^{-1} U$, where $U_n:=u(t_n+k\mathbf c)$.

Finally we note that we call a Runge-Kutta method stiffly accurate, if it satisfies
$\rkb^\top \rkA^{-1}=\mathbf e_m^\top:=(0,\dots,0,1)$. 
Stiffly accurate methods satisfy (we use that $\rkA\mathbf 1=\rkc$, see~(\ref{eq:subeqs:stage_order_condition}))
\begin{equation}\label{eq:1.4}
c_m=\rkb^\top\rkA^{-1}\rkc=\rkb^\top\rkA^{-1}\rkA\mathbf 1=
\rkb^\top\mathbf 1=1,
\end{equation}
and $r(\infty)=0.$ 
For stiffly accurate methods, the computation of the discrete stage derivative of a vector of stage of samples of a continuous function is particularly simple:

\begin{lemma}\label{lemma:2.7}
For stiffly accurate RK methods, the sequence
$G:=\partial^k F$ with $F_n=F(t_n+k\rkc)$, satisfies
\[
G_n=k^{-1}\rkA^{-1}(F(t_n+k\rkc)-\ones F(t_n)).
\] 
\end{lemma}

\begin{proof}
For stiffly accurate methods we have $r(\infty)=0$ and therefore
\[
\delta(z)=\rkA^{-1}-z \rkA^{-1} \ones\rkb^\top\rkA^{-1}
=\rkA^{-1} -z \rkA^{-1} \ones \mathbf e_m^\top.
\]
However, since $c_m=1$, we have $\mathbf e_m^\top F(t_{n-1}+ k\rkc)=F(t_{n-1}+k c_m)=F(t_n),$
which proves the result.
\end{proof}

We also make the following optional assumption, which allows to increase the convergence order in some cases.
\begin{assumption}  
   \label{ass:1.3} For all $t\in \R, t \neq 0$ the stability function satisfies $\abs{r(\ii t)}< 1$ and $r(\infty)<1$.
\end{assumption}

%%% MAIN RESULTS

\section{Error estimates}
\label{sect:error_estimates}
We are now in a position to formulate the main results of this article and put them into context with previous
results, most notably from~\cite{mallo_palencia_optimal_orders_rk}.

To simplify notation, we will write for $v\in \mathcal C([0,T];\mathcal X_\mu)$ with $\mu \ge 0,$
\[
\| v\|_{T,\mu}:=\max_{\tau\in [0,T]} \| v(\tau)\|_{\mathcal X_\mu}.
\]
For
functions $f:[0,T]\to \mathcal Y$, we will write $(\partial^{-1} f)(t):=\int_0^t f(\tau)\mathrm d\tau$,  where $\mathcal{Y}$ denotes a generic Banach space.

\subsection{The estimates of Alonso-Mallo and Palencia}
\label{sect:summary_of_AMP}
The following two results summarize the results of Alonso-Mallo and Palencia \cite{mallo_palencia_optimal_orders_rk}, rewritten with the notation of the present paper. The `proofs' which we provide clarify how notation needs to be adapted and how the hypotheses of the main results of \cite{mallo_palencia_optimal_orders_rk} are satisfied in our context.  

\begin{proposition}[{\cite[Theorem 1]{mallo_palencia_optimal_orders_rk}}]
\label{prop:AMP1}
  Let Assumption~\ref{ass:1.1} hold and
  assume that the exact solution $u$ satisfies $u \in \Cpspace{p+1}$ for some $\mu \geq 0$.
  Let $\{ u^k_n\}$ denote the Runge-Kutta approximation from~\eqref{eq:1.2}.
  There exist constants $k_0 > 0$, $C>0$, such that for $0<k\leq k_0$ and $0<  n k\leq T$ the following estimate holds:
  \begin{align}
    \label{eq:rk_approx_error_est}
    \| u(t_n) - u^k_n\|_{\mathcal{X}}
    &\leq C T \rho_{k}(T) k^{\min\{q+\mu,p\}} \Big(\sum_{\ell=q+1}^{p}
      \| u^{(\ell)}\|_{T,\mu}  + \| u^{(p+1)}\|_{T,0}\Big) .
  \end{align}
  The constant $C$ depends on the Runge-Kutta method,
  $\mu$, and the constants $M$ and $\omega$ from \eqref{eq:residual_estimate_semigroups}.
  The constant $k_0$ depends only on  $\omega$ and the Runge-Kutta method. 
\end{proposition}
\begin{proof}
  We only make  remark on the differences in notation. A different
  definition of interpolation spaces is given in \cite{mallo_palencia_optimal_orders_rk}, but the proof only relies on estimates of the form~\eqref{eq:interpolation_thm_est}.
  The choice of $k_0$ follows from the fact that it is only needed to ensure that $(I - k\,\rkA \otimes A)$ is invertible, see \eqref{eq:2.23}. The assumption $\mu \leq p-q$ in \cite[Theorem 1]{mallo_palencia_optimal_orders_rk} can be replaced by using the rate $\min\{p,q+\mu\}$ in \eqref{eq:rk_approx_error_est}
  as the spaces $\mathcal{X}_{\mu} \subseteq \mathcal{X}_{p-q}$ are nested for $\mu \geq p-q$.
  We also lowered the regularity requirements on the highest derivative compared to their
    stated result. The fact that this holds true follows from inspection of the proof. Compare
    also to Lemma~\ref{lemma:5.8} for the key ingredient.
\end{proof}

For a subset of Runge-Kutta methods, these estimates can be improved:

\begin{proposition}[{\cite[Theorem 2]{mallo_palencia_optimal_orders_rk}}]
\label{prop:AMP2}
  Let the assumptions of Proposition~\ref{prop:AMP1} hold and assume that, in addition,
  the RK-method satisfies Assumption~\ref{ass:1.3}.
  There exist constants $k_0 > 0$, $C>0$ such that for $0<k\leq k_0$ and $0<  n k\leq T$ the following improved estimate holds:
  \begin{align}
    \label{eq:rk_approx_error_est_improved}
    \| u(t_n) - u^k_n\|_{\mathcal{X}}
    &\leq C (1+T) \rho_{k}(T) k^{\min\{ q+ \mu+1,p\}} \sum_{\ell=q+1}^{p+1} 
    \|u^{(\ell)}\|_{T,\mu}.
  \end{align}
  The constant $C$ depends on the Runge-Kutta method,
  $\mu$, and the constants $M$ and $\omega$ from \eqref{eq:residual_estimate_semigroups}; $k_0$ depends only on the  constant $\omega$ and the Runge-Kutta method. 
\end{proposition}

\begin{proof}
  Again, this is just a reformulation of \cite[Theorem 2]{mallo_palencia_optimal_orders_rk}. 
  We first note that, due to our assumption on $r(\infty)$, we are always in the case $m=0$
  of \cite{mallo_palencia_optimal_orders_rk}.
  Since we assumed that on the imaginary axis $\abs{r(\ii t)}<1$ for $0\neq t \in \R$ ,
  we directly note that for  sufficiently small $k\leq k_0$, all the zeros
  of $r(z)-1$ except $z=0$ satisfy $\mathrm{Re}\, z > k \omega$.
  By the resolvent bound~\eqref{eq:residual_estimate_semigroups} we  can therefore estimate
 \[
 \norm{(z I - k A)^{-1}}_{\mathcal X \to\mathcal X}\leq \frac{M}{\mathrm{Re}\,z - k_0 \omega}, \qquad \mbox{if }\mathrm{Re}\,z < k_0\omega,
 \] 
i.e.,
   we have a uniform resolvent bound in  the set $Z_{\alpha,\delta}$ in \cite{mallo_palencia_optimal_orders_rk}.
  We also note that we reformulated the convergence rate such that we do not have the restriction $\mu \leq p-q-1$, since the 
  exceptional cases are already covered by Proposition~\ref{prop:AMP1}.
\end{proof}

\begin{remark}
  The assumption $\abs{r(z)}<1$ for $\Re(z) \leq 0$ and $r(\infty)\neq 1$ is satisfied by the Radau~IIA family of Runge-Kutta methods, but is
  violated by the Gauss methods, which satisfy $\abs{r(z)}=1$ on the imaginary axis.
\end{remark}

\subsection{New results in this article}
\label{sect:new_results}
In this section, we  present some a priori estimates for the convergence of Runge-Kutta methods when applied to
the abstract problem~\eqref{eq:1.1}. These can be seen as a continuation of \cite{mallo_palencia_optimal_orders_rk},
to the case where the boundary conditions are not given exactly but stem from computing discrete
integrals and differentials using the same Runge-Kutta method.

\begin{theorem}[Integrated estimate]\label{theorem:3.1}
Let $u$ solve \eqref{eq:1.1} with $u_0=0$, assume that for some $\mu \ge 0$ we have
\begin{align*}
u &\in \mathcal C^{p}([0,T];\mathcal X_\mu),  \quad
\lifting \Xi, F  \in \mathcal  C^{p-1}([0,T];\mathcal X_\mu) \cap C^{p}([0,T];\mathcal X_0).
\end{align*}
and let $x:=\partial^{-1} u$.
Let $U^k=\{ U^k_n\}$ and let $u^k=\{ u^k_n\}$ be the discrete approximation given by \eqref{eq:1.2} for a method satisfying Assumption \ref{ass:1.2}. If $X^k:=(\partial^k)^{-1} U^k$ and we define $x^k=\{x^k_n\}$ with the recurrence
\[
x_0^k:=0, \qquad x_{n+1}^k:=r(\infty) x_n^k + \rkb^\top \rkA^{-1} X_n^k, 
\]
then there exists a constant $k_0>0$ such that
 for all $k<k_0$ and  $n \in \N$ with $nk\le T$ the following estimate holds:
\begin{multline}
  \| x(t_n)-x^k_n\|_{\mathcal X}
  \le 
  C T \rho_k(T) k^{\min\{q+\mu+1,p\}}
  \begin{aligned}[t]\Big[
  \sum_{\ell=q}^{p-1}
  &\left( \| u^{(\ell)}\|_{T,\mu} + \|\lifting \Xi^{(\ell)}\|_{T,\mu}
    + \| F^{(\ell)}\|_{T,\mu}\right) \\
  +&\left( \| u^{(p)}\|_{T,\mu} + \|\lifting \Xi^{(p)}\|_{T}
    + \| F^{(p)}\|_{T}\right)   \Big]
  \end{aligned}\\
  + C \,T^2 \rho_k(T) k^{p}\left(
    \|u^{(p)}\|_{T,\mu} \|\lifting \Xi^{(p)}\|_{T}
  + \| F^{(p)}\|_{T} \right).
\end{multline}
If Assumption \ref{ass:1.3} holds
and if we assume the stronger regularities
\[
u\in \mathcal C^{p+1}([0,T];\mathcal X_\mu), \quad
F \in \mathcal C^{p}([0,T];\mathcal X_\mu), \quad
\lifting \Xi\in \mathcal C^{p}([0,T];\mathcal X_\mu),
\]

then
\begin{multline*}
\| x(t_n)-x^k_n\|_{\mathcal X}
\le  C (1+T) \rho_k(T) k^{\min\{q+\mu+2,p\}} \!
\begin{aligned}[t]
        \Big[\sum_{\ell=q}^{p}
        \| u^{(\ell)}\|_{T,\mu}+\|\lifting \Xi^{(\ell)}\|_{T,\mu}
        \hfill + \| F^{(\ell)}\|_{T,\mu} \\
        + \| u^{(p+1)}\|_{T,\mu} \Big]
      \end{aligned} \\
      + C \,T^2 \rho_k(T) k^{p}\left( \|\lifting \Xi^{(p)}\|_{T}
        + \| F^{(p)}\|_{T} \right).
\end{multline*}

The constant $k_0$ depends only on $\omega$ from~(\ref{eq:residual_estimate_semigroups})
  and the Runge-Kutta method. If $\omega=0$ then $k_0$ can be chosen arbitrarily large.
  $C$ depends on $\omega$, $M$ from~(\ref{eq:residual_estimate_semigroups}), the Runge-Kutta method
  and $\mu$.
\end{theorem}

\begin{theorem}[Differentiated estimate]\label{theorem:3.2}
  Let $u$ solve \eqref{eq:1.1} with $u_0=0$ and
  assume $\dot{u}(0)=0$.
  Assume that for some $\mu \ge 0$ we have
  \[
    u\in \mathcal C^{p+1}([0,T];\mathcal X_\mu)\cap C^{p+2}([0,T];\mathcal X_0) 
   , \quad
   \lifting \Xi, F \in \mathcal C^{p}([0,T];\mathcal X_\mu)
   \cap C^{p+1}([0,T];\mathcal X_0),
  \]
  and let $v:=\dot u$.
  Let $U^k=\{ U^k_n\}$ and $u^k=\{ u^k_n\}$ be the discrete approximation given by \eqref{eq:1.2} for a stiffly accurate method satisfying Assumption \ref{ass:1.2}.
  
  If $V^k:=\partial^kU^k$ and $v^k_n=\mathbf e_m^\top V^k_{n-1}$,
  then there exists a constant $k_0>0$ such that
    for all $k<k_0$ and  $n\ge 1$ such that $nk\le T$ the following estimate holds:
  \begin{align*}
    \| v(t_n)-v^k_n\|_{\mathcal X}
	\le  C T \rho_k(T) k^{\min\{q+\mu,p\}-1} 
        \Big(\!&\sum_{\ell=q+1}^{p}
    \big( \| u^{(\ell+1)}\|_{T,\mu}+\|\lifting \Xi^{(\ell)}\|_{T,\mu}
      +\| F^{(\ell)}\|_{T,\mu}\big) \\
        &+ \| u^{(p+2)}\|_{T,0}+\|\lifting \Xi^{(p+1)}\|_{T,0}
      +\| F^{(p+1)}\|_{T,0} \Big).
  \end{align*}
  If, in addition, the method satisfies Assumption \ref{ass:1.3}
   and
      \[
          u\in \mathcal C^{p+2}([0,T];\mathcal X_\mu)
          , \quad
          \lifting \Xi, F \in \mathcal C^{p+1}([0,T];\mathcal X_\mu)
          \cap C^{p+2}([0,T];\mathcal X_0),
  \]
    
   then
   \begin{align*}
    \| v(t_n)-v^k_n\|_{\mathcal X}
	\le  C (1+T) \rho_k(T) k^{\min\{q+\mu,p\}} 
    &\Big(\sum_{\ell=q+1}^{p+1}
     \big( \| u^{(\ell)}\|_{T,\mu}+\|\lifting \Xi^{(\ell)}\|_{T,\mu}
      +\| F^{(\ell)}\|_{T,\mu}\big) \\
     &+ \| u^{(p+2)}\|_{T,\mu}+\|\lifting \Xi^{(p+2)}\|_{T,0}
      +\| F^{(p+2)}\|_{T,0} \Big)
   \end{align*} 
  The constant $k_0$ depends only $\omega$ from~(\ref{eq:residual_estimate_semigroups})
    and the Runge-Kutta method. If $\omega=0$, then $k_0$ can be chosen arbitrarily large.
    $C$ depends on $\omega$, $M$ from~(\ref{eq:residual_estimate_semigroups}), the Runge-Kutta method
  and $\mu$.
  
\end{theorem}

\begin{theorem}[Strong norm estimate]\label{theorem:3.3}
  If $u$, $\Xi$, and $F$ satisfy the hypotheses of Theorem \ref{theorem:3.2}, and $\{u_n^k\}$ is the approximation provided by \eqref{eq:1.2} for a stiffly accurate method satisfying Assumption \ref{ass:1.2} (same restrictions with respect to $k_0$), then for $n$ such that $nk\le T$,
  there holds
    \begin{align*}
    \| A_\star(u(t_n)-u^k_n)\|_{\mathcal X}
	\le  C T \rho_k(T) k^{\min\{q+\mu,p\}-1} 
        \Big(\!&\sum_{\ell=q+1}^{p}
    \big( \| u^{(\ell+1)}\|_{T,\mu}+\|\lifting \Xi^{(\ell)}\|_{T,\mu}
      +\| F^{(\ell)}\|_{T,\mu}\big) \\
        &+ \| u^{(p+2)}\|_{T,0}+\|\lifting \Xi^{(p+1)}\|_{T,0}
      +\| F^{(p+1)}\|_{T,0} \Big).
  \end{align*}
If, in addition, the method satisfies Assumption \ref{ass:1.3}, then
\begin{align*}
  \| A_\star(u(t_n)-u^k_n)\|_{\mathcal X}
  \le  C (1+T) \rho_k(T) k^{\min\{q+\mu,p\}} 
  &\Big(\sum_{\ell=q+1}^{p+1}
    \big( \| u^{(\ell)}\|_{T,\mu}+\|\lifting \Xi^{(\ell)}\|_{T,\mu}
    +\| F^{(\ell)}\|_{T,\mu}\big) \\
  &+ \| u^{(p+2)}\|_{T,\mu}+\|\lifting \Xi^{(p+2)}\|_{T,0}
    +\| F^{(p+2)}\|_{T,0} \Big)
\end{align*}
  
$C$ depends on $\omega$, $M$ from~(\ref{eq:residual_estimate_semigroups}), the Runge-Kutta method
  and $\mu$.
\end{theorem}

\begin{remark}
  Most of the effort in proving the above theorem is in order to obtain a convergence rate higher than $q$, even though the
  constraint in the stages only approximate with order $q$. This is possible by exploiting the additional structure
  of the discretization error of the side constraint.
\end{remark}

\begin{remark}
  We formulated all our results for homogeneous initial conditions, since it is sufficient for our purposes in time domain BEM and convolution quadrature.
  It should be possible to generalize these results to the case of $u_0 \in \dom(\AA^{s})$ for sufficiently large $s \geq 1$ by considering the
  evolution of the semigroup with inhomogeneous side-constraint but homogeneous initial condition and the semigroup of homogeneous constraint but inhomogeneous
  $u_0$ separately.
\end{remark}
\begin{remark}
  The loss of order by $1$ in Theorem~\ref{theorem:3.3} compared to Propositions~\ref{prop:AMP1} and~\ref{prop:AMP2} is to be expected. Indeed,  if
  we look at the case $u \in \dom(A^{\mu})$ for $\mu\geq 1$,
  this means
  $\AAstar u \in \dom(A^{\mu-1})$. Applying Proposition~\ref{prop:AMP2} to this semigroup then also gives 
  a reduced order of $k^{\min(q+\mu,p)}$.
\end{remark}

%% SOME COMPUTATIONS

\section{Some computations related to the main theorems}
\label{sect:computations}

We will collect the sampled data and the stage and step parts of the solutions in four formal series
\begin{subequations}
\begin{alignat}{6}
\label{eq:Zdata}
\widehat F^k(z) &:=\sum_{n=0}^\infty F(t_n+k\rkc)z^n, 
	 && \qquad &
\widehat\Xi^k(z) &:=\sum_{n=0}^\infty \Xi(t_n+k\rkc)z^n,\\
\label{eq:Zsol}
\widehat U^k(z) &:=\sum_{n=0}^\infty U^k_n z^n,
	&&&
\widehat u^k(z)&:=\sum_{n=0}^\infty u^k_n z^n.
\end{alignat}
\end{subequations}
If the data functions are polynomially bounded in time, the series in \eqref{eq:Zdata} are convergent (in $\mathcal X^m$ and $\mathcal M^m$ respectively) with at least unit radius of convergence. Because of the equivalent formulation of the numerical method in the form \eqref{eq:1.3}, and using \eqref{eq:2.23},
it follows that for $k\le k_0$ (with $k_0$ chosen using \eqref{eq:2.22}), the numerical solution is
at least bounded in the form $\norm{U^k_n}_{\mathcal{X}} \lesssim C^{n}$. Thus, the two series in \eqref{eq:Zsol} also converge
on a sufficiently small disk. 

%For the first group of arguments, these series can be used formally.

\begin{proposition}\label{prop:3.1}
The sequences $\{ U^k_n\}$ and $\{u^k_n\}$ satisfy equations \eqref{eq:1.2} if and only if
\begin{subequations}\label{eq:3.3}
\begin{alignat}{6}
\label{eq:3.3a}
k^{-1} \delta(z) \widehat U^k(z)
	&=(\mathcal I\otimes A_\star) \widehat U^k(z)+ \widehat F^k(z)
		+\frac{k^{-1}}{1-r(\infty) z} \rkA^{-1}\ones  u_0,\\
\label{eq:3.3b}
(\mathcal I\otimes B)\widehat U^k(z) &=\widehat\Xi^k(z),\\
\label{eq:3.3c}
\widehat u^k(z) &=\frac{z}{1-r(\infty) z} \rkb^\top \rkA^{-1} \widehat U^k(z)+\frac1{1-r(\infty)z}u_0^k.
\end{alignat}
\end{subequations}
\end{proposition}

\begin{proof} 
Let us start by proving a simple result: {\em the discrete equations \eqref{eq:1.2a} and \eqref{eq:1.2c} hold if and only if \eqref{eq:1.2a} and 
\begin{equation}\label{eq:3.4}
u^k_{n+1}=r(\infty) u^k_n+\rkb^\top\rkA^{-1}  U^k_n
\end{equation}
hold.} To see this, note that \eqref{eq:1.2a} is equivalent to
\[
\rkA^{-1}(U^k_n-\ones u^k_n)
	=k((\mathcal I\otimes A_\star)U^k_n+F(t_n+k\rkc))
\]
and therefore \eqref{eq:1.2a} and \eqref{eq:1.2c} imply
\[
\rkb^\top\rkA^{-1} (U^k_n-\ones u^k_n)
	=k((\rkb^\top\otimes A_\star)U^k_n+\rkb^\top F(t_n+k\rkc))
	=u^k_{n+1}-u^k_n,
\]
or equivalently \eqref{eq:3.4}. The reciprocal statement is proved similarly.
The recurrence \eqref{eq:3.4} is equivalent to \eqref{eq:3.3c}. At the same time, the recurrence \eqref{eq:1.2a} is equivalent to
\begin{equation}\label{eq:3.5}
k^{-1}\rkA^{-1} (\widehat U^k(z)-\ones \widehat u^k(z))
	=(\mathcal I\otimes A_\star)\widehat U^k(z)+\widehat F^k(z). 
\end{equation}
Plugging \eqref{eq:3.3c} into \eqref{eq:3.5}, the formula \eqref{eq:3.3a} follows.
\end{proof}

Proposition \ref{prop:3.1} is a rephrasing of \cite[Lemma 3.19]{rieder_diss},
where the computation is also laid out in more detail.
Note how equations \eqref{eq:3.3a}-\eqref{eq:3.3b} relate strongly to \eqref{eq:1.1}, with the discrete symbol $k^{-1}\delta(z)$ playing the role of the time derivative and 
\[
\frac{k^{-1}}{1-r(\infty)z} \rkA^{-1}\ones
\]
playing the role of a discrete Dirac delta at time $t=0$.

\begin{lemma}[{\cite[Lemma~{2.6}]{BanLM}}]
\label{prop:rk_spectrum_of_delta}
If the matrix $\rkA$ of the RK method is invertible, then for $\abs{z} < 1$
\[
\sigma(\delta(z)) \subseteq \sigma(\rkA^{-1}) \cup \{ w \in \C: r(w) z = 1 \}.
\]
In particular, if the Runge-Kutta method is A-stable (Assumption \ref{ass:1.2}), then $\sigma(\delta(z))\subset \C_+$.
\end{lemma}

We need a  corollary to the previous result:
\begin{corollary}
  \label{cor:spectrum_delta_uniform}
  Let Assumption~\ref{ass:1.2} hold.
  Then, for all $r_0<1$, there exists a constant $d>0$ such that for all $\abs{z}<r_0$ it holds that
  $$
  \sigma\big(\delta(z)\big) \subset \Big\{ w \in \C_+: \Re(w) > d \Big \}.
  $$
\end{corollary}
\begin{proof}
  
    In view of of  Lemma~\ref{prop:rk_spectrum_of_delta},
    since $\sigma(\rkA)$ is finite, independent of $z$, and contained in $\C_+$, we are mainly
    concerned with the set $\{ w \in \C: r(w) z = 1 \}$.
    We first note that
    $$
    \bigcup_{\abs{z}\leq r_0} \{ w \in \C: r(w) z = 1 \}
    \subseteq \{ w \in \C: \abs{r(w)} \geq  1/r_0 \}.
    $$
    Second, we observe that by taking $d_0$ small enough, we can ensure that
    $w \mapsto r(w)$ is continuous for $\Re(w)\leq d_0$ and thus
    \begin{align*}
      \{ w \in \C: \abs{r(w)} \geq  1/r_0 \} \cap \{w \in \C: \Re(w)\leq d_0\} \\
      = {r|_{\{\Re(w)\leq d_0\}}}^{-1}\big([1/r_0,\infty)\big) 
    \end{align*}
    is a closed set.
    Third, by considering the limit along the imaginary axis, we get
    $$
    \abs{r(\infty)}=\lim_{n \to \infty} \abs{r(\ii n)}\leq 1.
    $$
    Thus, for $\abs{w}$ sufficiently large, it holds that
    $\abs{r(w)} \leq 1/r_0$.

    Overall, we get that
    $$
    \{ w \in \C: \abs{r(w)} \geq  1/r_0 \} \cap \{w \in \C: \Re(w)\leq d_0\} 
    $$
    is a compact set with empty intersection  with the imaginary axis.
    Thus, it must have a positive distance from it.
    Applying these observations and Lemma~\ref{prop:rk_spectrum_of_delta} concludes the proof.

\end{proof}

\begin{lemma}\label{lemma:3.6}	
  Let Assumptions \ref{ass:1.1} and \ref{ass:1.2} hold. For $r_0<1$,
  there exists $k_0=k_0(\omega,r_0)>0$ such that for all $k\le k_0$ and $|z|\le r_0$ the problem
\begin{subequations}\label{eq:3.9}
\begin{alignat}{6}
-k^{-1}\delta(z) \widehat U
	 + (\mathcal I\otimes A_\star) \widehat U &=\widehat F,\\
(\mathcal I\otimes B) \widehat U &=\widehat\Xi
\end{alignat}
\end{subequations}
has a unique solution for arbitrary $\widehat F\in \mathcal X^m$ and $\widehat\Xi\in \mathcal M^m$. If $\omega=0$ in Proposition \ref{prop:2.1}, then there are no restrictions on $k$, and the results holds for all $|z|<1$.
\end{lemma}

\begin{proof}
Assume first that $\mathcal S\in \mathbb C^{m\times m}$ is such that $\sigma(\mathcal S) \subset \{ z\,:\, \mathrm{Re}\,z>\omega\}$ and consider the problem
\begin{subequations}\label{eq:3.10}
\begin{alignat}{6}
-(\mathcal S \otimes I) \widehat U + (\mathcal I\otimes A_\star) \widehat U &=\widehat F,\\
(\mathcal I\otimes B) \widehat U &=\widehat\Xi.
\end{alignat}
\end{subequations}
Take first $\widehat V:=(\mathcal I\otimes \mathscr E)\widehat\Xi$ (where $\mathscr E$ is the lifting operator of Assumption \ref{ass:1.1}) and then seek $\widehat W\in (\mathrm{dom}\,A)^m$ satisfying
\[
-(\mathcal S\otimes I)\widehat W+(\mathcal I\otimes A)\widehat W
=\widehat F+((\mathcal S-\mathcal I)\otimes I)\widehat V.
\]
This problem is uniquely solvable by Lemma~\ref{lemma:2.2}, since $\sigma(A)\subset \{ z\,:\,\mathrm{Re}\,z\le \omega\}$ and therefore $\sigma(A)\cap\sigma(\mathcal S)=\emptyset.$ We then define $\widehat U:=\widehat V+\widehat W$, which solves \eqref{eq:3.10}.
To see uniqueness, one observes that the difference of two solutions solves the
  homogeneous problem ($\widehat{\Xi}=0$ and  $\widehat{F}=0$) for
  which uniqueness was established in Lemma~\ref{lemma:2.2}.

By Corollary~\ref{cor:spectrum_delta_uniform},
the union of the spectra of $\delta(z)$ for $|z|\le r_0$ has a positive distance $d(r_0)>0$ from the imaginary axis.
If we take $k_0< d(r_0)/\omega$, then $\sigma(k^{-1}\delta(z))\subset \{ s\,:\, \mathrm{Re}\,s>\omega\}$ for all $|z|\le r_0$ and $k\le k_0$. When $\omega=0$, we can take any $k_0$. By the previous considerations this implies unique solvability.
\end{proof}

\begin{proposition}\label{prop:3.7}
Let $U^k=\{U^k_n\}$ and $u^k=\{u^k_n\}$ be sequences satisfying \eqref{eq:1.2} with $u^k_0=0$. The sequence $V^k=\{ V^k_n\}=\partial^k U^k$ satisfies
\begin{subequations}\label{eq:3.8}
\begin{align}
  V^k_n &= \ones  \, v^k_n+ k (\rkA \otimes \AAstar) V^k_n + k \rkA G_n^k,  \\
  (\mathcal I\otimes B ) V^k_n&=\Theta_n^k, \\
  v^k_{n+1} &= r(\infty) v^k_n + \rkb^\top\rkA^{-1} V^k_n,
\end{align}
\end{subequations}
for data $v^k_0=0$, $G^k=\{G^k_n\}:=\partial^k \{ F(t_n+k\rkc)\}$, and $\Theta^k=\{ \Theta^k_n\}:=\partial^k \{ \Xi(t_n+k\rkc)\}$. Moreover,
\begin{subequations}
\begin{alignat}{6}
\label{eq:3.90a}
(\mathcal I\otimes A_\star) U^k_n &=V^k_n-F(t_n+k\rkc),\\
\label{eq:3.90b}
V^k_n &=k^{-1}\rkA^{-1}(U^k_n-\ones u^k_n).
\end{alignat}
\end{subequations}
\end{proposition}

\begin{proof}
Recall that equations \eqref{eq:1.2a} and \eqref{eq:1.2c} are equivalent to \eqref{eq:1.2a} and \eqref{eq:3.4}, as shown in the proof of Proposition \ref{prop:3.1}. Moreover, the latter equations are equivalent to \eqref{eq:3.3} in the $Z$-domain. In the present case we have $u_0=0$. For a given square matrix $\mathcal P\in \mathbb C^{m\times m}$ and an operator $C$, we have 
\[
(\mathcal P\otimes I)(\mathcal I\otimes C)=\mathcal P\otimes C=(\mathcal I\otimes C)(\mathcal P\otimes I),
\]
which proves that
\begin{subequations}\label{eq:3.100}
\begin{alignat}{6}
k^{-1} \delta(z) \widehat V^k(z)
	&=(\mathcal I\otimes A_\star) \widehat V^k(z)+ \widehat G^k(z),\\
(\mathcal I\otimes B)\widehat V^k(z) &=\widehat\Theta^k(z),\\
\widehat v^k(z) &=\frac{z}{1-r(\infty) z} \rkb^\top \rkA^{-1} \widehat V^k(z).
\end{alignat}
\end{subequations}
By Proposition \ref{prop:3.1}, equations \eqref{eq:3.100} are equivalent to \eqref{eq:3.8}. Finally \eqref{eq:3.90a} follows from \eqref{eq:3.3a}, while \eqref{eq:3.90b} follows from \eqref{eq:3.5} and \eqref{eq:3.90a}.
\end{proof}

\begin{proposition}\label{prop:44.5}
Let $U^k=\{U^k_n\}$ and $u^k=\{u^k_n\}$ be sequences satisfying \eqref{eq:1.2} with $u^k_0=0$. The sequence $X^k=\{X^k_n\}=(\partial^k)^{-1} U^k$ satisfies
\begin{subequations}
\begin{align}
  X^k_n &= \ones  \, x^k_n+ k (\rkA \otimes \AAstar) X^k_n + k \rkA H_n^k,  \\
  (\mathcal I\otimes B ) X^k_n&=\Gamma_n^k, \\
  x^k_{n+1} &= r(\infty) x^k_n + \rkb^\top\rkA^{-1} X^k_n\\
  		&= x^k_n+k(\rkb^\top\otimes A_\star)X^k_n+k \rkb^\top H_n^k,
\end{align}
\end{subequations}
for data $x^k_0:=0$,  $H^k=\{H^k_n\}:=(\partial^k)^{-1} \{ F(t_n+k\rkc)\}$, and $\Gamma^k=\{\Gamma^k_n\}:=(\partial^k)^{-1} \{ \Xi(t_n+k\rkc)\}$. 
\end{proposition}

\begin{proof}
Follow the proof of Proposition \ref{prop:3.7}.
\end{proof}

\section{Some Lemmas regarding Runge-Kutta methods}
\label{sect:some_lemmas}

In order to shorten the statements of the results of this section, in all of them we will understand that:
\begin{itemize}
\item[(1)] We have an RK method with coefficients $\rkA, \rkb, \rkc$ satisfying Assumption \ref{ass:1.2} (invertibility of $\rkA$ and A-stability). The method has classical order $p$ and stage order $q$.
\item[(2)] We have an operator $A$ in $\mathcal X$ that is the generator of a $C_0$-semigroup,
  characterized by the quantities $M$ and $\omega$ of Proposition \ref{prop:2.1}.
  The associated Sobolev tower $\{ \mathcal X_\mu\}$, obtained by interpolation of $\dom A^\mu$ for positive integer values of $\mu$, will also be used. 
\end{itemize}
The following lemma will be used at a key point in the arguments below.

\begin{lemma}\label{lemma:2.4}
Let $A$ be a linear operator in $\mathcal X$ and $q$ be a rational function bounded at infinity with poles outside $\sigma(A)$. The following properties hold:
\begin{itemize}
\item[(a)] The operator $q(A)$ maps $\dom A^\ell$ to $\dom A^\ell$ for all $\ell.$ 
\item[(b)] If $0\notin \sigma(A)$, and we define $p(z):=z^{-\ell} q(z)$, then $q(A)=p(A)A^\ell$ in $\dom A^\ell$.
\end{itemize}
\end{lemma}

\begin{proof} 
To prove (a), show first by induction on $\ell $ that $(A-\lambda\id)^{-1}$ maps $\dom A^\ell$ into $\dom A^{\ell+1}$. Using this result for each of the factors in the definition \eqref{eq:2.4} the result follows. To prove (b) note first that $p$ is rational, bounded at infinity, and that $\sigma(A)$ does not intersect the set of poles of $p$. Using Definition \ref{def:2.3}, we have $p(A)=q(A)A^{-\ell}=A^{-\ell} q(A)$, and the result follows. 
\end{proof}

We start by recalling some simple facts about RK methods that we will need in the sequel. Using the notation $\rkc^{\ell}:=(c_{1}^{\ell},\dots, c_{m}^{\ell})^\top$, the following equalities (order conditions) hold (see e.g.~\cite{mallo_palencia_optimal_orders_rk,ostermann_roche})):
\begin{subequations}
 \label{eq:subeqs:stage_order_condition}
\begin{align}
   \rkc^{\ell} &=\ell \rkA \rkc^{\ell-1},  \quad & 0 &\leq 1\leq \ell \leq q, \label{eq:stage_order_condition} \\
  \rkb^\top \rkA^{j} \rkc^{\ell} &=\frac{\ell!}{(j+\ell+1)!}, \quad& 0 &\leq j+\ell \leq p-1.\label{eq:order_condition}
\end{align}
\end{subequations}
Therefore, 
\begin{subequations}\label{eq:5.2}
\begin{alignat}{6}
\label{eq:5.2a}
\rkb^\top\rkA^j(\rkc^\ell-\ell \rkA \rkc^{\ell-1}) &= 0
	& \qquad & 0\le j\le p-\ell-1, \quad 1\leq \ell\le p\\
\label{eq:5.2b}
\ell \rkb^\top\rkc^{\ell-1} &=1, && 1 \leq \ell\le p.
\end{alignat}
\end{subequations}
For a stiffly accurate method we have \eqref{eq:1.4} and therefore
\begin{equation}\label{eq:4.44}
\rkb^\top\rkA^{-1}\rkc^\ell=c_m^\ell=1.
\end{equation}

\begin{lemma}[Discrete antiderivative and RK quadrature]\label{lemma:5.2}
Let $f:[0,T]\to \mathcal X$, $g:=\partial^{-1}f$, $G^k=\{ G^k_n\}=(\partial^k)^{-1} \{ f(t_n+k\rkc)\}$ and $\{ g^k_n\}$ be given by the recursion
\[
g^k_0:=0, \qquad g_{n+1}^k:=g_n^k+k\rkb^\top f(t_n+k\rkc).
\]
For the errors $d^k_n:=g(t_n)-g^k_n$, and for $n$ such that $nk\le T$, we have the estimates
\begin{subequations}
\begin{alignat}{6}
\label{eq:55.4a}
\| d^k_n\|_{\mathcal X} 
	&\le C T k^p  \| f^{(p)}\|_{T,0}, \\
\label{eq:55.4b}
\| d^k_n-d^k_{n-1}\|_{\mathcal X}
	& \le C k^{p+1} 
		\max_{t_{n-1}\le t \le t_n}\| f^{(p)}(t)\|_{\mathcal X}.
\end{alignat}
Additionally, at the stage level we have
\begin{equation}
\label{eq:55.4c}
\| k\rkb^\top g(t_n+k\rkc)-k\rkb^\top G^k_n\|_{\mathcal X}
	\le C k^{p+1} \big(\| f^{(p-1)}\|_{T,0}+T \| f^{(p)}\|_{T,0}\big).
\end{equation}
\end{subequations}
\end{lemma}

\begin{proof}
Since
\[
d^k_n-d^k_{n-1}
	=\int_{t_{n-1}}^{t_n} f(\tau)\mathrm d\tau-k\rkb^\top f(t_{n-1}+k\rkc)
\]
and the RK method defines a quadrature formula with degree of precision $p-1$ (as follows from the order conditions \eqref{eq:5.2b}), the bound \eqref{eq:55.4b} follows. Using a telescopic sum argument (and $d^k_0=0$), \eqref{eq:55.4a} is shown to be a consequence of \eqref{eq:55.4b}. Since $G^k_n=\ones g^k_n+k\rkA f(t_n+k\rkc)$ (see Lemma \ref{lemma:2.6}), we can write
\[
k\rkb^\top (g(t_n+k\rkc)-G^k_n)
 = k\rkb^\top (g(t_n+k\rkc)-\ones g(t_n)-k\rkA \dot g(t_n+k\rkc))
 	-k  d^k_n.
\]
The order conditions \eqref{eq:5.2a} with $j=0$ prove that if $\pi$ is a polynomial of degree $p-1$, then
\[
\rkb^\top (\pi(k\rkc)-\ones \pi(0)-k\rkA \dot\pi(k\rkc))=0.
\]
A Taylor expansion of degree $p-1$ about $t_n$ can then be used to show that
\[
\|k\rkb^\top (g(t_n+k\rkc)-\ones g(t_n)-k\rkA \dot g(t_n+k\rkc))\|_{\mathcal X}
	\le k^{p+1} \max_{t_n\le \tau\le t_{n+1}} \| g^{(p)}(\tau)\|_{\mathcal X},
\]
and \eqref{eq:55.4c} follows.
\end{proof}

\subsection{Estimates on rational functions of the operator}

We will use the rational functions
\begin{alignat}{6}
\label{eq:4.1}
r_{\ell,\beta}(z) & := z \rkb^\top(\id-z\rkA)^{-1}\rkA^\beta (\rkc^\ell-\ell \rkA\rkc^{\ell-1}),
\qquad \beta\in \{-1,0,1\}, \\
\label{eq:55.6}
s_n(z) &:=\sum_{j=0}^n r(z)^j.
\end{alignat}
Note that these rational functions are bounded at infinity and that $r_{\ell,\beta}(0)=0$. We will also use the vector-valued rational function
\begin{equation}\label{eq:5.defg}
\mathbf g(z)^\top:=z \rkb^\top (\id - z \rkA)^{-1},
\end{equation}
and note that $\mathbf g(0)=\mathbf 0$ and $r(z)=1+\mathbf g(z)^\top\ones$.

\begin{lemma}\label{lemma:4.1}
The rational functions \eqref{eq:4.1} satisfy
\begin{equation}
  \label{eq:lemma:4.1}
r_{\ell,\beta}(z)=\mathcal O(|z|^{p+1-\ell-\beta}) \quad \mbox{as $|z|\to 0$},
	\qquad \ell\le p, \quad \beta\in\{0,1\}.
\end{equation}
The estimate~\eqref{eq:lemma:4.1} is also valid for $\beta=-1$ if the method is stiffly accurate.
\end{lemma}

\begin{proof}
  Using the Neumann series expansion for the inverse in the definition or $r_{\ell,-1}$, valid for $|z|\le 1/\|\rkA\|$, we can write
\begin{alignat*}{6}
r_{\ell,-1}(z) &= z (\rkb^\top\rkA^{-1}\rkc^\ell-\ell \rkb^\top\rkc^{\ell-1})
	+\sum_{j=0}^\infty z^{j+2} \rkb^\top \rkA^j (\rkc^\ell-\ell\rkA \rkc^{\ell-1})\\
	&=\sum_{j=p-\ell}^\infty z^{j+2} \rkb^\top \rkA^j (\rkc^\ell-\ell\rkA \rkc^{\ell-1})
	=\mathcal O(|z|^{p-\ell+2}),
\end{alignat*}
after using \eqref{eq:5.2} and \eqref{eq:4.44} (for which we needed stiff accuracy). With the same technique we write
\[
r_{\ell,0}(z)=\sum_{j=0}^\infty z^{j+1}\rkb^\top \rkA^j (\rkc^\ell-\ell\rkA \rkc^{\ell-1})
=\sum_{j=p-\ell}^\infty z^{j+1}\rkb^\top \rkA^j (\rkc^\ell-\ell\rkA \rkc^{\ell-1})
=\mathcal O(|z|^{p-\ell+1}),
\] 
and
\[
r_{\ell,1}(z)=\sum_{j=1}^\infty z^j \rkb^\top \rkA^j (\rkc^\ell-\ell\rkA \rkc^{\ell-1})
=\sum_{j=p-\ell}^\infty z^j \rkb^\top \rkA^j (\rkc^\ell-\ell\rkA \rkc^{\ell-1})
=\mathcal O(|z|^{p-\ell}),
\]
which finishes the proof.
\end{proof}

\begin{lemma}\label{lemma:4.2}
If the RK method satisfies Assumption \ref{ass:1.3}, then there exists a constant $k_0>0$ depending on the RK method and on $\omega$ such that for $\beta\in\{0,1\}$, $\ell \leq p-\beta$, and all $0<k\leq k_0$ with $0\leq n\,k \leq T$, we have the estimate
\begin{equation}\label{eq:4.4}
  \norm{ s_n(k A) r_{\ell,\beta}(k\,\AA)}_{\mathcal X_\mu\to\mathcal X}
  \leq C \,\rho_k(T) k^{\min\{\mu,p-\ell-\beta\}}
\end{equation}
with $\rho_k(T)$ defined in~(\ref{eq:def_rho_k}).
If $\ell=p$ and $\beta=1$, the left-hand side of \eqref{eq:4.4} is bounded by $C\rho_k(T)$. 
The constant $C>0$ in \eqref{eq:4.4} depends only on the Runge-Kutta method,  $M$ and $\omega$, $k_0$, $\ell$, and $\mu$, but is independent of  $n$ and $k$.  If the Runge-Kutta method is stiffly accurate, the estimate \eqref{eq:4.4} also holds for $\beta=-1$.
If $\omega=0$, then $k_0$ can be chosen arbitrarily.
\end{lemma}

\begin{proof}
We adapt the proof of \cite[Lemma 6]{mallo_palencia_optimal_orders_rk}, which only covers the case $\beta=0$. Consider first the case $p-\ell-\beta\ge 0$ and take any integer $\mu$ such that $0\le \mu\le p-\ell-\beta$. We then write
\[
r_{\ell,\beta}(z) \sum_{j=0}^n r(z)^j 
=(r(z)^{n+1}-1) q_{\ell,\beta,\mu}(z) z^\mu,
	\quad 
q_{\ell,\beta,\mu}(z):=\frac{r_{\ell,\beta}(z)}{(r(z)-1) z^\mu}.
\]
From its definition, we observe that the rational function $r_{\ell,\beta}$
  is bounded at infinity, has a zero of order $p-\ell-\beta+1$ at $z=0$ (Lemma \ref{lemma:4.1}), and its poles are contained in $\sigma(\rkA^{-1})\subset \C_+$ (by A-stability).
Note that $r(0)=1$ with nonzero derivative, since $r(z)$ is an approximation of $\exp(z)$ with order $p+1$. Therefore, the rational function $(r(z)-1)z^\mu$ has a zero of order $\mu+1\le p-\ell-\beta+1$ at $z=0$. All its other zeros are in $\mathbb C_+$, since by Assumption \ref{ass:1.3}, $|r(z)|<1$ for all $z\neq 0$ with $\mathrm{Re}\,z\le 0$. This implies that the rational function $q_{\ell,\beta,\mu}$, which is bounded at infinity (we are using here that $r(\infty)\neq 1$, which is part of Assumption \ref{ass:1.3}), has its poles in 
\begin{equation}\label{eq:Lambda}
\Lambda:=\{ z\neq 0\,:\, r(z)=1\}\cup\sigma(\rkA^{-1})\subset \C_+.
\end{equation}
Take now $k_0>0$ such that $k_0\omega< \lambda_{\min}:=\min\{ \mathrm{Re}\,\lambda\,:\,\lambda\in \Lambda\}$ and note that by Proposition \ref{prop:2.1}
\[
\| (\lambda \id-k A)^{-1}\|_{\mathcal X\to\mathcal X} \le \frac{M}{\lambda_{\min}-k_0\omega}
\qquad\forall\lambda\in \Lambda, \quad \forall k\le k_0.
\]
Therefore, using \eqref{eq:2.3}
\begin{equation}\label{eq:4.5}
\| q_{\ell,\beta,\mu}(kA)\|_{\mathcal X \to \mathcal X} \le C \qquad k\le k_0,
\end{equation}
where $C$ depends on $M$, $\omega$, $k_0$, and the RK method. By Lemma \ref{lemma:2.4} we have
\[
r_{\ell,\beta}(kA) \sum_{j=0}^n r(kA)^jx=k^\mu (r(kA)^{n+1}-\id) q_{\ell,\beta,\mu}(kA) A^\mu x
\quad\forall x\in \dom A^\mu, \quad k\le k_0.
\] 
This and \eqref{eq:4.5} proves \eqref{eq:4.4} for integer $\mu\le p-\ell-\beta$. For larger integer values of $\mu$, the result does not need to be proved
as the maximum rate is already attained. We just have to estimate
  the $\XX_{p-\ell-\beta}$ norm by the stronger $\XX_\mu$ norm.
For real values of $\mu$, we use interpolation.

We still need to prove the result when $p-\ell-\beta=-1,$ which can only happen when $\ell=p$ and $\beta=1.$
We note that $r_{p,1}(0)=0$ and
  we can therefore argue as in the previous case for $\mu=0$.
\end{proof}

\begin{lemma}\label{lemma:4.3}
If the RK method satisfies Assumption \ref{ass:1.3} and $k_0$ is the value given in Lemma \ref{lemma:4.2},
then
\begin{equation}\label{eq:4.7}
\|  s_n(k \,\AA) \mathbf g(k\, \AA)^\top\|_{\mathcal{X}^m\to\mathcal X}\leq  C\, \rho_k(T),
\end{equation}
for all $k\le k_0$ and $n$ such that $nk\le T$. 
\end{lemma}

\begin{proof}
Since $\mathbf g(0)=\boldsymbol 0$, we can adapt the proof of Lemma \ref{lemma:4.2} to each of the components of the vector-valued function $\mathbf g$. The key step is to show that $\mathbf h(z)^\top:=(r(z)-1)^{-1} \mathbf g(z)$ is bounded at infinity and has all its poles in the set defined in \eqref{eq:Lambda} and therefore
\[
\| \mathbf h(k\,A)^\top\|_{\mathcal X^m \to\mathcal X}\le C \qquad \forall k\le k_0.
\]
Since the operator
$s_n(k \,\AA) \mathbf g(k\, \AA)^\top$
on the left-hand side of \eqref{eq:4.7} can be rewritten as $(r(kA)^{n+1}-\id)\mathbf h(k\,A)^\top$, the bound \eqref{eq:4.7} follows readily. 
\end{proof}

When dealing with Runge-Kutta methods that do not satisfy the additional Assumption~\ref{ass:1.3},
we still have the following result:

\begin{lemma}\label{lemma:4.4}
For $k_0>0$ taken as in Lemma \ref{lemma:4.2}, we can bound for all $k\le k_0$
\begin{equation}\label{eq:4.100}
\| r_{\ell,\beta}(k\,A)\|_{\mathcal X_\mu\to\mathcal X}\le C k^{\min\{\mu,p+1-\ell-\beta\}}
\end{equation}
for $\ell\le p$, $\beta\in \{0,1\}$ and $\mu \ge 0$. The constant $C$ depends on $M$, $\omega$, $k_0$, $\mu$, and the RK method. The estimate~(\ref{eq:4.100}) also holds for $\beta=-1$ if the method is stiffly accurate. Additionally
\begin{equation}\label{eq:4.8}
\|\mathbf g(k A)^\top\|_{\mathcal X^m\to\mathcal X}\le C, \qquad k\le k_0.
\end{equation}
\end{lemma}

\begin{proof}
The argument to prove \eqref{eq:4.100} is very similar to that of Lemma \ref{lemma:4.2}. By interpolation it is clear that we just need to prove the result for any integer $\mu$ satisfying $0\le\mu\le p+1-\ell-\beta$. Consider then the rational function $q_{\ell,\beta,\mu}(z):=z^{-\mu}r_{\ell,\beta}(z)$, which is bounded at infinity and has all its poles in $\sigma(\mathcal Q^{-1})$ (see \eqref{eq:Lambda}). We can then use the same argument to prove \eqref{eq:4.5} for this redefined new function $q_{\ell,\beta,\mu}$. (Note that we do not use Assumption \ref{ass:1.3} in this argument.)  Using that $r_{\ell,\beta}(k\,A)=k^\mu q_{\ell,\beta,\mu}(k A) A^\mu$ in $\dom A^\mu$, the result follows. Stiff accuracy of the method is used in the case $\beta=-1$ when we apply Lemma \ref{lemma:4.1}, dealing with the zeros of $r_{\ell,-1}$.

The proof of \eqref{eq:4.8} is a similar adaptation of the proof of Lemma \ref{lemma:4.3}. 
\end{proof}

\subsection{Estimates on discrete convolutions}

The RK error will naturally induce several types of discrete convolutions that we will need to estimate separately. In all of them we will have the structure
\begin{equation}\label{eq:55.15}
\omega_0 =0, \qquad
	\omega_{n+1}:=r(kA)\omega_n+k\eta_n,\quad n\ge 0. 
\end{equation}
We first deal with the simplest cases.

\begin{lemma}\label{lemma:5.7}
For $nk\le T$, the sequence defined by \eqref{eq:55.15} can be bounded by
\[
\| \omega_n\|_{\mathcal X}\le C T \rho_k(T) \max_{j\le n} \| \eta_j\|_{\mathcal X}.
\]
If $\eta_n:=\mathbf g(kA)^\top \boldsymbol\xi_n$ for $\boldsymbol\xi_n\in \mathcal X^m$, then
\[
\|\omega_n\|_{\mathcal X}\le C T \rho_k(T) \max_{j\le n} \|\boldsymbol\xi_n\|_{\mathcal X^m}.
\]
\end{lemma}

\begin{proof}
The sequence defined by the recurrence \eqref{eq:55.15} can be written as the discrete convolution
\begin{equation}\label{eq:55.16}
\omega_{n+1}=k\sum_{j=0}^n r(kA)^{n-j}\eta_j.
\end{equation}
The estimates then follow from the definition of $\rho_k(T)$ and from \eqref{eq:4.8}.
\end{proof}

The next estimate is related to the consistency error of the RK method in the sense of how the RK method approximates derivatives at the stage level. We introduce the operator
\begin{equation}\label{eq:4.9}
\boldsymbol D^k(y;t):=y(t+k\rkc)-y(t)\mathbf 1-k \rkA \dot y(t+k\rkc).
\end{equation}

\begin{lemma}\label{lemma:4.5}
If $y\in \mathcal C^{p+1}([0,T];\mathcal X)$, then
\begin{equation}\label{eq:4.10}
\boldsymbol D^k(y;t)=
\sum_{j=q+1}^p \frac{k^j}{j!} (\rkc^j-j\rkA\rkc^{j-1}) y^{(j)}(t)+\boldsymbol R^k(t),
\end{equation}
where
\begin{equation}\label{eq:4.11}
\| \boldsymbol R^k(t)\|_{\mathcal X^m}\le C k^{p+1} \max_{t\le \tau\le t+k} \| y^{(p+1)}(\tau)\|_{\mathcal X}.
\end{equation}
\end{lemma}

\begin{proof}
We will prove the result for scalar-valued functions, since the extension for Banach-space valued functions is straightforward. We use a Taylor expansion
\[
y(t+\tau)=\sum_{j=0}^p \frac{\tau^j}{j!} y^{(j)}(t)+r(\tau;t),
\] 
where
\[
r(\tau;t):=\frac{\tau^{p+1}}{p!}\int_0^1 (1-s)^p y^{(p+1)}(t+\tau s)\mathrm ds.
\]
Plugging this decomposition in the definition of $\boldsymbol D^k(y;t)$, we obtain
\[
\boldsymbol D^k(y;t)
	= \sum_{j=1}^p \frac{k^j}{j!}  (\rkc^j-j\rkA\rkc^{j-1}) y^{(j)}(t)
		+ r(k\rkc;t)-k\rkA \dot r(k\rkc;t)
\]
Defining $\boldsymbol R^k(t):=r(k\rkc;t)-k\rkA \dot r(k\rkc;t)$ and using \eqref{eq:subeqs:stage_order_condition}, we have \eqref{eq:4.10}. Since
\[
  \dot r(\tau;t)=
  \frac{\tau^p}{(p-1)!}
   \int_0^1 (1-s)^{p-1}y^{(p+1)}(t+\tau s)\mathrm ds \Big),
\]
the bound \eqref{eq:4.11} is straightforward.  
\end{proof}

We are almost ready for the two main lemmas of this section, the first one without Assumption \ref{ass:1.3} and the second one with it. These results and their proofs follow \cite[Theorem 1 and 2]{mallo_palencia_optimal_orders_rk}, where only the case $\beta=0$ is covered. 

\begin{lemma}\label{lemma:5.8}
  Let $y\in\mathcal C^{p+1}([0,T];\mathcal X)\cap \mathcal C^p([0,T];\mathcal X_\mu)$, and we consider the sequence $\omega_n$ defined by the recurrence \eqref{eq:55.15} with $\eta_n:=k^{-1}\mathbf g(kA)^\top (k\rkA)^\beta \boldsymbol D(y;t_n)$. Then there exists a constant $k_0>0$ depending
    only on $\omega$ from~(\ref{eq:residual_estimate_semigroups}) and the RK-method such that
    for $n k \leq T$
\[
\|\omega_n\|_{\mathcal X}
\le C T \rho_k(T) k^{\min\{q+\mu+\beta,p+\beta,p\}} 
	\left(\sum_{j=q+1}^p \| y^{(j)}\|_{T,\mu}
		+\| y^{(p+1)}\|_{T,0}\right),
\]
for $\beta\in \{0,1\}$, $\mu \ge 0$, and $k\le k_0$. The estimate also holds for  $\beta=-1$ if the method is stiffly accurate.
If $\omega=0$, then $k_0$ can be chosen arbitrarily.
\end{lemma}

\begin{proof}
Considering the functions
\begin{equation}\label{eq:4.16}
e^k_\beta(t)
	:=\mathbf g(kA)^\top (k\rkA)^\beta \boldsymbol D^k(y;t), \qquad \beta\in \{-1,0,1\},
\end{equation}
we have that $\eta_n=k^{-1}e^k_\beta(t_n)$.
Using Lemmas \ref{lemma:4.4} and \ref{lemma:4.5} (recall the definition of $r_{j,\beta}$ in \eqref{eq:4.1}), we can bound
\begin{alignat*}{6}
\| e^k_\beta(t)\|_{\mathcal X}
	\le & \sum_{j=q+1}^p \frac{k^{j+\beta}}{j!} 
		\norm{r_{j,\beta}(kA)y^{(j)}(t)}_{\mathcal X}
		+k^\beta\|\mathbf g(kA)^\top \rkA^\beta \boldsymbol R^k(t)\|_{\mathcal X} \\
	\lesssim & \sum_{j=q+1}^p k^{\min\{\mu+j+\beta,p+1\}}  
	\| y^{(j)}(t)\|_{\mathcal X_\mu}
			+ k^{\beta} \| \boldsymbol R^k(t)\|_{\mathcal X^m}.
\end{alignat*}
By \eqref{eq:4.11}, we then have
\begin{equation}\label{eq:4.166}
\| e^k_\beta(t)\|_{\mathcal X}
	\le C k^{1+\min\{q+\mu+\beta,p+\beta,p\}} 
	\left(\sum_{j=q+1}^p \| y^{(j)}(t)\|_{\mathcal X_\mu}
		+\max_{t\le \tau\le t+k}\| y^{(p+1)}(\tau)\|_{\mathcal X}\right),
\end{equation}
and  the result then follows from Lemma \ref{lemma:5.7}.
\end{proof}

\begin{lemma}\label{lemma:5.9}
  Let the RK method satisfy Assumption \ref{ass:1.3}. Let $y\in\mathcal \mathcal C^{p+1}([0,T];\mathcal X_\mu)$.
  Then there exists a constant $k_0>0$ depending
    only on $\omega$ from~(\ref{eq:residual_estimate_semigroups}) and the RK-method such that
  the sequence $\omega_n$ defined in Lemma \ref{lemma:5.8} satisfies
\[
\| \omega_n\|_{\mathcal X}
\le  C (1+T) \rho_k(T) k^{\min\{q+\mu+\beta+1,p\}} 
	\left(\sum_{j=q+1}^{p+1} \| y^{(j)}\|_{T,\mu} \right)
\]
for $\beta\in \{0,1\}$, $\mu \ge 0$, and $k\le k_0$.
If the method is stiffly accurate and $y\in\mathcal C^{p+2}([0,T];\mathcal X)\cap \mathcal C^{p+1}([0,T];\mathcal X_\mu)$, then for $\beta=-1$
\[
\| \omega_n\|_{\mathcal X}
\le  C (1+T) \rho_k(T) k^{\min\{q+\mu,p\}} 
	\left(\sum_{j=q+1}^{p+1} \| y^{(j)}\|_{T,\mu}
		+\| y^{(p+2)}\|_{T,0}\right).
            \]
            If $\omega=0$, then $k_0$ can be chosen arbitrarily.
\end{lemma}

\begin{proof}
We will use the function $e^k_\beta$ defined in \eqref{eq:4.16} and 
\begin{alignat}{6}
\nonumber
\omega_n=& \sum_{j=0}^n r(kA)^{n-j} e^k_\beta(t_j)\\
\label{eq:5.19}
= & s_n(kA) e^k_\beta(t_0) +\sum_{j=1}^n s_{n-j}(kA)(e^k_\beta(t_j)-e^k_\beta(t_{j-1}))  
	+ e^k_\beta(t_n),
\end{alignat}
an expression involving the rational functions $s_n$ defined in \eqref{eq:55.6} (recall that $s_0=1$).
We first apply Lemmas \ref{lemma:4.2}, \ref{lemma:4.3} and \ref{lemma:4.5} to estimate
\begin{alignat*}{6}
\| s_n(kA) e^k_\beta(t)\|_{\mathcal X}
	\le & \sum_{j=q+1}^p \frac{k^{j+\beta}}{j!} 
      \| s_n(kA)r_{j,\beta}(kA)\|_{\mathcal X_\mu\to\mathcal X} \| y^{(j)}(t)\|_{\mathcal X_\mu}
      \\
		& \qquad+ C k^\beta \| s_n(kA)\mathbf g(kA)^\top\|_{\mathcal X^m\to\mathcal X} 
			\|\boldsymbol R^k(t)\|_{\mathcal X^m}\\
	\lesssim & \rho_k(T) \sum_{j=q+1}^p k^{\min\{j+\mu+\beta,p\}}\| y^{(j)}(t)\|_{\mathcal X_\mu} \\
			 & \qquad + \rho_k(T) k^{\beta+p+1} \max_{t\le \tau\le t+k} \| y^{(p+1)}(\tau)\|_{\mathcal X}\\
	\lesssim & \rho_k(T) k^{1+\min\{q+\mu+\beta,p+\beta,p-1\}} 
			\left(\sum_{j=q+1}^p \| y^{(j)}(t)\|_{\mathcal X_\mu}
				+\max_{t\le \tau\le t+k} \| y^{(p+1)}(\tau)\|_{\mathcal X}\right).
\end{alignat*}
Since 
\[
e^k_\beta(t)-e^k_\beta(t-k)
=\mathbf g(kA)^\top (k\rkA)^\beta\boldsymbol D^k(y-y(\cdot-k);t),
\]
and using that $\norm{y^{j}(t)-y^{(j)}(t-k)}_{\XX_{\mu}} \leq k\max_{t-k\le \tau\le t+k}\norm{y^{(j+1)}}_{\XX_{\mu}}$,
the analogous computation to the above bound, but using $y-y(\cdot-k)$ as data  implies
\begin{alignat*}{6}
& \| s_n(kA)(e^k_\beta(t)-e^k_\beta(t-k)\|_{\mathcal X} \\
& \lesssim  \rho_k(T) k^{1+\min\{q+1+\mu+\beta,p\}}
		\left( \sum_{j=q+2}^{p+1} \max_{t-k\le \tau\le t} \| y^{(j)}(\tau)\|_{\mathcal X_\mu}
			+\max_{t-k\le \tau\le t+k}\| y^{(p+2)}(\tau)\|_{\mathcal X}\right),
                    \end{alignat*}
and therefore
\begin{alignat*}{6}
& \hspace{-1cm}
	\sum_{j=1}^n \| s_{n-j}(kA)(e^k_\beta(t_j)-e^k_\beta(t_{j-1}))\|_{\mathcal X} \\
	\lesssim & \rho_k(T) t_n k^{\min\{q+1+\mu+\beta,p\}}
		\left(\sum_{j=q+2}^{p+1} \| y^{(j)}\|_{t_n,\mu}
		+\| y^{(p+2)}\|_{t_{n+1},0}\right).
\end{alignat*}
Note that if $\beta\in \{0,1\}$ we can make a simpler estimate for
  the term  originating from $\mathbf{R}^k$, (i.e., the one containing the highest derivative)
  using less regularity for $y$ by not taking advantage of
the difference between $y^{(p+1)}(t_j)$ and $y^{(p+1)}(t_{j-1})$ and thus end up requiring less regularity.  
Using the estimate \eqref{eq:4.166} for the last term in \eqref{eq:5.19},
we have thereby already derived estimates for all three terms in \eqref{eq:5.19}. 
\end{proof}

\section{Proofs}
\label{sect:proofs}

The two different cases (with or without Assumption \ref{ass:1.3}) will be collected by using the parameter
\begin{equation}\label{eq:6.0}
\alpha:=
\begin{cases} 
  1, & \mbox{if Assumption \ref{ass:1.3} holds}, \\
  0, & \mbox{otherwise}.
\end{cases}
\end{equation}

\subsection{Proof of Theorem \ref{theorem:3.1}}
Recall that $u$ solves (\ref{eq:1.1}) with $u(0)=0$. The functions $\Xi$ and $F$ are
  the given boundary and volume data.
If $\Gamma:=\partial^{-1} \Xi$ and $H:=\partial^{-1}F$, then $x=\partial^{-1} u$ solves
\begin{align}
\dot x(t)=A_\star x(t)+H(t), \quad t>0, \qquad Bx(t)=\Gamma(t), \qquad x(0)=0. \label{eq:6.1}
\end{align}
On the other hand, $\{ X^k_n\}=(\partial^k)^{-1} \{ U^k_n\}$  solves by  Proposition~\ref{prop:44.5}:
\begin{subequations}
\begin{alignat}{6}
\label{eq:6.2a}
X_n^k =
	& \ones  x_n^k + k(\rkA\otimes A_\star)  X^k_n+ k\rkA H^k_n,\\
\label{eq:6.2b}
(I\otimes B) X_n^k =
	& \Gamma_n^k,\\
\label{eq:6.2c}
x_{n+1}^k =
	& x_n^k + k(\rkb^\top\otimes A_\star) X^k_n + k \rkb^\top H^k_n.
\end{alignat}
\end{subequations}

Before we can estimate the difference between the functions $x$ and $x_n^k$, we need one final lemma.
\begin{lemma}
  \label{lemma:u_with_boundary_corrections_is_regular}

  Let $x$ solve
  \begin{align}
    \dot x(t)=A_\star x(t)+H(t), \quad t>0, \qquad Bx(t)=\Gamma(t), \qquad x(0)=0. \label{eq:6.1_copy}
  \end{align}

  Assume that for some $\mu \ge 0$ we have
  \[
    x\in \mathcal C^{p+1}([0,T];\mathcal X_\mu), \quad
    H \in \mathcal C^p([0,T];\mathcal X_\mu), \quad
    \lifting \Gamma \in \mathcal C^p([0,T];\mathcal X_\mu).
  \]

  Then $x - \lifting \Gamma \in \mathcal C^p([0,T];\mathcal X_{\mu+1})$.
\end{lemma}
\begin{proof}
  We set $y:=x - \lifting \Gamma$.
  By assumption we have $y \in C^p([0,T];\mathcal X_{\mu})$
  and $B\left(x- \lifting \Gamma\right)=0$.
  Since
  $x \in \dom(\AA_\star)$ and $\operatorname{range}{\lifting}\subset \dom(\AA_\star)$
  this implies $y(t) \in \dom(\AA)$ for all $t \in [0,T]$.
  We further calculate using~\eqref{eq:6.1_copy}
  and $\operatorname{range}\lifting \subseteq \operatorname{ker}\id-\AA_\star$:
  \begin{align*}
    \AA y
    &= \AA_{\star} x - \AA_{\star} \lifting \Gamma
      = \dot{x} - H - \lifting \Gamma.
  \end{align*}
  Each of the contributions is assumed in $\mathcal C^p([0,T];\mathcal X_{\mu})$,
  thus  $y \in \mathcal C^p([0,T];\mathcal X_{\mu+1})$.
\end{proof}

 We will need the sequences $\{ \gamma_n^k\}$ and $\{ h_n^k\}$ with the scalar parts of the computations of $\{\Gamma^k_n\}$ and $\{H^k_n\}$ respectively,  namely (see Lemma \ref{lemma:2.6})
\begin{subequations}
\begin{alignat}{6}
\label{eq:6.3a}
\gamma^k_0:=0,& \qquad \gamma^k_n=\gamma^k_{n-1}+k\rkb^\top\Xi(t_n+k\rkc),\\
\label{eq:6.3b}
h^k_0:=0,& \qquad h^k_n=h^k_{n-1}+k\rkb^\top F(t_n+k\rkc).
\end{alignat}
\end{subequations}

We then consider
\[
\Delta_n^k:=(I\otimes \lifting)(\Gamma(t_n+k\rkc)-\Gamma_n^k),
	\qquad
\delta_n^k:=\lifting(\Gamma(t_n)-\gamma_n^k).
\]
 Using \eqref{eq:6.3a}, the definition $\Gamma^k=(\dd)^{-1} \Xi$, and~(\ref{eq:2.dantider}), we can write
\begin{equation}\label{eq:6.4}
\Delta_n^k-\ones \delta^k_n=(I \otimes \lifting) \Gamma(t_n+k\rkc)-\ones \lifting \Gamma(t_n)-k\rkA \otimes  \lifting  \dot\Gamma(t_n+k\rkc)
=\boldsymbol D^k(\lifting \Gamma;t_n).
\end{equation}
Lemma \ref{lemma:5.2} (take $f=\lifting\Xi$ for the first three inequalities and $f=F$ for the last one) proves that
\begin{subequations}
\begin{alignat}{6}
\label{eq:6.6a}
\|\delta^k_n\|_{\mathcal X}
	\le & C T k^p \| \lifting\Xi^{(p)}\|_{T,0},\\
\label{eq:6.6b}
\|\delta^k_n-\delta^k_{n-1}\|_{\mathcal X}
	\le & C k^{p+1} \|\lifting \Xi^{(p)}\|_{T,0},\\
\label{eq:6.6c}
\| k\rkb^\top \Delta^k_n\|_{\mathcal X}
\le & C k^{p+1} ( \|\lifting \Xi^{(p-1)}\|_{T,0}+T \|\lifting \Xi^{(p)}\|_{T,0}) \\
\label{eq:6.6d}
\| H(t_n)-h_n^k\|_{\mathcal X}
	\le & C T k^p \| F^{(p)}\|_{T,0}.
\end{alignat}
\end{subequations}
The error analysis is derived by tracking the evolution of the following differences
\[
E_n^k :=  x(t_n + k \rkc)-X^k_n-\Delta^k_n \in (\operatorname{dom} A)^m,
	\qquad
e^k_n :=  x(t_n)-x^k_n -\delta^k_n,
\]
(compare \eqref{eq:6.1} and \eqref{eq:6.2b} to see the vanishing boundary condition for $E_n^k$) and note that by \eqref{eq:6.6a}
\[
\| x(t_n)-x^k_n\|_{\mathcal X}
	\le \| e^k_n\|_{\mathcal X}+ C T k^p \| \lifting \Xi^{(p)}\|_{T,0},
\]
which shows that we only need to estimate $e^k_n$ to prove Theorem~\ref{theorem:3.1}.

We start with the observation that $x$ solves the following equation,
as can be easily derived from the equation~\eqref{eq:6.1}:
\begin{align}
  x(t_n + k \rkc)&=
                   \begin{multlined}[t]
                     \ones x(t_n) + k \rkA \otimes \AA_{\star}x(t_n+k \rkc)\\
                     + x(t_n+k \rkc) - k \rkA \dot{x}(t_n+k \rkc) + k \rkA H(t_n+k\rkc) - \ones x(t_n) 
                     \end{multlined} \nonumber\\\vspace{5mm}
                 &= \ones x(t_n) + k \rkA \otimes \AA_{\star}x(t_n+k \rkc) + k\rkA H(t_n+k\rkc) + \boldsymbol D^{k}(x,t_n) .
                   \label{eq:exact_solution_solves_rk_type_eqn}
\end{align}

Recalling that Assumption \ref{ass:1.1} included the hypothesis $\mathrm{range}\,\lifting\subset \ker (I-A_\star)$, we have $(\rkA\otimes A_\star)\Delta^k_n=\rkA \Delta^k_n$.
Combining \eqref{eq:exact_solution_solves_rk_type_eqn} and \eqref{eq:6.2a}, we get
\begin{align*}
  E^k_n&=\ones e^k_n+ k(\rkA\otimes A)E^k_n
  + \boldsymbol D^k(x,t_n) - k \rkA\big(H^k_n - H(t_n + k \rkc)\big)
  +\ones \delta^k_n-\Delta^k_n+k\rkA\Delta^k_n.		
\end{align*}
Naive estimation of the terms $\boldsymbol D^k(x,t_n)$ and $\Delta^k_n - \ones \delta^k_n$
would yield convergence rates similar to Propositions~\ref{prop:AMP1} and~\ref{prop:AMP2}.
In order to get an increased rate, as stated in Theorem~\ref{theorem:3.1}, we combine these two terms using the function
$Y(t):=x(t) - \lifting \Gamma(t)$. By Lemma~\ref{lemma:u_with_boundary_corrections_is_regular}
and the assumptions of Theorem~\ref{theorem:3.1}, this function satisfies
$Y \in \mathcal{C}^{p+\alpha}([0,T];\XX_{\mu+1}) \cap \mathcal{C}^{p+1}([0,T];\XX_{0})$.

We can thus further simplify
\begin{align}
  \label{eq:proof_thm3.1_err_rep_est1}
  E^k_n&=\begin{multlined}[t]
    \ones e^k_n+ k(\rkA\otimes A)E^k_n 
         + \boldsymbol D^k(x,t_n) - \boldsymbol D^k(\lifting \Gamma,t_n) \\
         + k \rkA \boldsymbol D^k(H,t_n) - k \rkA \ones\big( h^k_h - H(t_n)\big)
         +k\rkA\Delta^k_n \end{multlined}\\
       &=\begin{multlined}[t]
         \ones e^k_n+ k(\rkA\otimes A)E^k_n  +  \boldsymbol D^k(Y,t_n) \\
         +k \rkA \boldsymbol D^k(H,t_n) - k \rkA \ones \big(h^k_h - H(t_n)\big)
         +k\rkA\Delta^k_n.\end{multlined}
\end{align}

This then immediately gives (recall \eqref{eq:5.defg} for the definition of $\mathbf g$)
\begin{equation}
  \label{eq:proof_thm3.1_err_rep_est2}
k(\rkb^\top\otimes A)E^k_n
	=\mathbf g(kA)^\top
    \left[\ones e^k_n
    + \boldsymbol D^k(Y,t)
    + k \rkA \boldsymbol D^k(H,t_n)
    - k \rkA \ones \big(h^k_h - H(t_n)\big)
    +k\rkA\Delta^k_n\right].
\end{equation}

It is easy to see from \eqref{eq:6.1} that $x$ satisfies
$$
x(t_{n+1})= x(t_n) + k \rkb^\top \otimes A_{\star} x(t_n + k \,\rkc)  + \left[x(t_{n+1}) - x(t_n) - k \rkb^\top \dot{x}(t_{n}+ k\rkc) + k \rkb^\top H(t_n+k \rkc)\right] .
$$
Subtracting \eqref{eq:6.2c} from this,
plugging in~\eqref{eq:proof_thm3.1_err_rep_est2}, using that $(\rkb^\top \otimes A_\star)\Delta^k_n=\rkb^\top \Delta^k_n$,
and setting
$$\varphi^k_n:=\left[x(t_{n+1}) -x(t_n) - k \rkb^\top \dot{x}(t_{n}+k\rkc)\right]
+k \rkb^\top(H(t_n + k\rkc) - H^k_n),$$
we have 
\begin{alignat*}{6}
e^k_{n+1}
	=& e^k_n+k(\rkb^\top\otimes A)E^k_n 
        + k(\rkb^\top \otimes A_\star) \Delta^k_n+\delta^k_n-\delta^k_{n+1} + \varphi^k_n
        \\
	=& r(kA)e^k_n 		
    +\mathbf g(kA)^\top (k\rkA)\Delta^k_n
    +\mathbf g(kA)^\top \boldsymbol D^k(Y,t_n) \\
    &+ \mathbf g(kA)^\top (k\rkA)\boldsymbol D^k(H,t_n) 
    - \mathbf g(kA)^\top (k\rkA) \ones \big(h^k_h - H(t_n)\big) \\
    &+ k\rkb^\top  \Delta^k_n+\delta^k_n-\delta^k_{n+1}
    +\varphi^k_n.
\end{alignat*} 
What is left is the careful combination of terms so that we can bound everything using Lemmas \ref{lemma:5.7}, \ref{lemma:5.8}, and \ref{lemma:5.9} by writing
\begin{alignat*}{6}
  e^k_{n+1}-r(kA) e_n^k
  = &
  \mathbf g(kA)^\top (\boldsymbol D^k(Y;t_n)) \\
  &+\mathbf g(kA)^\top (k\rkA) (\boldsymbol D^k(\lifting \Gamma + H;t_n)) \\    
  &+\mathbf g(kA)^\top \rkA\ones \, k \big(\delta_n^k - \big(h^k_h - H(t_n)\big)\big) \\
  &+ k\rkb^\top \Delta_n^k + (\delta_n^k-\delta_{n+1}^k) + \varphi^k_n
\end{alignat*}
Since the above recurrence defining $\{e^k_n\}$ is linear as a function of the right-hand side, we can estimate its norm by adding the effects of each of the terms.
In the order in which they appear in the last expression, we use:
Lemmas \ref{lemma:5.8}-\ref{lemma:5.9} with $\beta=0$, but noting that $Y(t) \in \dom(A^{\mu+1})$;
Lemmas \ref{lemma:5.8}-\ref{lemma:5.9} with $\beta=1$; Lemma \ref{lemma:5.7}
combined with \eqref{eq:6.6a} and \eqref{eq:6.6d}; Lemma \ref{lemma:5.7} combined with \eqref{eq:6.6c} and \eqref{eq:6.6b};
for the first term of $\varphi^k_h$ we use Lemma~\ref{lemma:5.7} combined with Lemma~\ref{lemma:5.2} with $f:=\dot{x}$.
Finally, for the second contribution to $\varphi^k_h$, we use~(\ref{eq:55.4c}).

Combined, these results give
\begin{alignat*}{6}
\| e_n^k\|_{\mathcal X}
	\le & C T \rho_k(T) k^{\min\{q+\mu+1+\alpha,p\}}
		\Big(\sum_{\ell=q+1}^{p} \| Y^{(j)}\|_{\mu+1,T}
					+\| Y^{(p+1)}\|_{\alpha(\mu+1),T}\Big) \\
		& + C T \rho_k(T) k^{\min\{q+\mu+1+\alpha,p\}}
		\Big(\sum_{\ell=q+1}^{p} \| (\lifting \Gamma)^{(j)}\|_{\mu,T}
        +\| (\lifting \Gamma)^{(p+1)}\|_{\alpha \mu,T}\Big) \\
        & + C T \rho_k(T) k^{\min\{q+\mu+1+\alpha,p\}}
		\Big(\sum_{\ell=q+1}^{p} \| H^{(j)}\|_{\mu,T}
					+\| H^{(p+1)}\|_{\alpha \mu,T}\Big) \\
        &+ CT^2 \rho_k(T) k^{p+1} \| \lifting\Xi^{(p)}\|_{0,T}
        + C T^2 \rho_k(T) k^{p+1} \|{F^{(p)}}\|_{0,T}\\
        & + CT\rho_k(T) k^{p+1} 
        \big( \|\lifting \Xi^{(p-1)}\|_{T,0}+T \|\lifting \Xi^{(p)}\|_{T,0}\big) \\
        & + C T\rho_k(T) k^{p+1}  
          \big( \|H^{(p-1)}\|_{T,0}+T \|H^{(p)}\|_{T,0}\big).
\end{alignat*}
If we apply Lemma~\ref{lemma:u_with_boundary_corrections_is_regular} to bound the $\mathcal{X}_{\mu+1}$-norm, we arrive at the stated estimate.

%%%%%%%%%%%%%%% 

\subsection{Proof of Theorem \ref{theorem:3.2}}

This proof is very similar to the one for Theorem \ref{theorem:3.1}, while slightly simpler. We will point out the main steps of the proof. Note that we use the
simple form of $\partial^k$ for stiffly accurate RK methods given in Lemma \ref{lemma:2.7}. 
We define $G:=\dot F$ and  $\Theta:=\dot\Xi$ so that $v=\dot u$ satisfies
\[
\dot v(t)=A_\star v(t)+G(t), \quad t>0, \qquad Bv(t)=\Theta(t), \qquad v(0)=0.
\]
Its RK approximation 
\begin{subequations}
\begin{alignat}{6}
\label{eq:6.7a}
\widetilde V_n^k =
	& \ones \widetilde v_n^k + k(\rkA\otimes A_\star)\widetilde V^k_n
		+ k\rkA G(t_n+k\rkc),\\
\label{eq:6.7b}
(I\otimes B) \widetilde V_n^k =
	& \Theta (t_n+k\rkc),\\
\label{eq:6.7c}
\widetilde v_{n+1}^k =
& \widetilde v_n^k + k(\rkb^\top\otimes A_\star) \widetilde V^k_n
+ k \rkb^\top G(t_n + k \rkc),
\end{alignat}
\end{subequations}
and $\{V_n^k\}=\partial^k \{U_n^k\}$ satisfies (see Proposition \ref{prop:3.7} and Lemma \ref{lemma:2.7}, where we use stiff accuracy of the RK scheme, and recall that $\{G^k_n\}=\partial^k \{F(t_n+k\rkc)\}$ and $\{\Theta^k_n\}=\partial^k\{ \Xi(t_n+k\rkc)\}$)
\begin{subequations}
\begin{alignat}{6}
\label{eq:6.8a}
V_n^k =
	& \ones  v_n^k + k(\rkA\otimes A_\star)  V^k_n+ k\rkA G^k_n,\\
\label{eq:6.8b}
(I\otimes B) V_n^k =
	& \Theta_n^k=k^{-1}\rkA^{-1}(\Xi(t_n+k\rkc)-\ones \Xi(t_n)),\\
\label{eq:6.8c}
v_{n+1}^k =
& v_n^k + k(\rkb^\top\otimes A_\star) V^k_n
+ k \rkb^\top G_n^k.
\end{alignat}
\end{subequations}
Let then
\[
\Delta^k_n:=(\mathcal I\otimes \lifting)(\Theta_n^k-\Theta(t_n+k\rkc))
=k^{-1}\rkA^{-1} \boldsymbol D^k(\lifting\Xi;t_n)
\]
and (note \eqref{eq:6.7b} and \eqref{eq:6.7c})
\[
E^k_n :=  V^k_n-\widetilde V^k_n-\Delta^k_n \in (\operatorname{dom} A)^m,
	\qquad
e^k_n :=  v^k_n-\widetilde v^k_n.
\]
By \eqref{eq:6.7a} and \eqref{eq:6.8a}, using that $(\rkA\otimes A_\star)\Delta_n^k=\rkA\Delta^k_n$ (assumption on the lifting) and Lemma \ref{lemma:2.7} to represent $G_n^k$, we have
\[
k(\rkb^\top\otimes A) E^k_n
	=\mathbf g(kA)^\top 
		(\ones e_n^k-\Delta_n^k+k\rkA\Delta_n^k+\boldsymbol D^k(F;t_n))
\]
and therefore, from \eqref{eq:6.7c} and \eqref{eq:6.8c}
\begin{multline}
e^k_{n+1}=r(kA) e^k_n 
	-\mathbf g(kA)^\top (k\rkA)^{-1}\boldsymbol D^k(\lifting\Xi;t_n) 
        +\mathbf g(kA)^\top\boldsymbol D^k(\lifting\Xi+F;t_n) \\
        + k \rkb^\top \rkA^{-1} D^k(\lifting \Xi+F;t_n).
      \end{multline}
      The final term can be shown using to be of order $\bigO(k^{p+1})$
        by combining~(\ref{eq:4.10}) with~(\ref{eq:5.2b}) and~(\ref{eq:4.44}).
      
      Use then Lemmas \ref{lemma:5.8} and \ref{lemma:5.9} with $\beta=-1$ and $\beta=0$
      as well as Lemma~\ref{lemma:5.7} to bound
\begin{alignat*}{6}
\| e_n^k\|_{\mathcal X}
\le & C T \rho_k(T) k^{\alpha-1+\min\{q+\mu,p\}}
\left(\sum_{j=q+1}^{p+\alpha} \|\lifting\Xi^{(j)}\|_{T,\mu}+
  \|\lifting\Xi^{(p+1+\alpha)}\|_{T,0}\right) \\
&+C T \rho_k(T) k^{\alpha-1+\min\{q+\mu,p\}}
\left(\sum_{j=q+1}^{p+\alpha} \|F^{(j)}\|_{T,\mu}+
  \| F^{(p+1+\alpha)}\|_{T,0}\right). 
\end{alignat*}
Finally Propositions \ref{prop:AMP1} and \ref{prop:AMP2} are used to bound 
\begin{equation}
\| v(t_n)-\widetilde v_n^k\|_{\mathcal X}
	\le C T \rho_k(T)k ^{\min\{ q+\mu+\alpha,p\}} 
        \Big(
        \sum_{\ell=q+2}^{p+1+\alpha} \| u^{(\ell)}\|_{T,\mu}
        + \| u^{(p+2)}\|_{T,0}\Big).
\end{equation}

\subsection{Proof of Theorem \ref{theorem:3.3}}

Thanks to Proposition \ref{prop:3.7}, Theorem \ref{theorem:3.3} can be proved as an easy corollary of Theorem \ref{theorem:3.2}. Since the last stage of a stiffly accurate
method is the step, we have that \eqref{eq:3.90a} implies that
\[
A_\star u^k_n=v^k_n-F(t_n)
\]
and therefore
\[
A_\star u(t_n)-A_\star u_n^k=v(t_n)-v^k_n.
\]

\section{Maximal dissipative operators in Hilbert space}
\label{sect:maximal_dissipative_operators}
In this short section
we summarize some results that show that the hypotheses on the abstract equation and its discretization are simpler for maximal dissipative operators on Hilbert spaces.
These results are well-known and will be needed when applying the theory developed in the previous sections to some model problems in Section~\ref{sect:applications}.

If $A$ is \emph{maximal dissipative} in the Hilbert space $\mathcal X$, i.e.,
\[
\mathrm{Re} \langle\AA x,x\rangle_{\mathcal{X}} \leq 0 
	\qquad\forall x\in \dom \AA,
\]
and if $\AA-\id:\dom\AA\to \mathcal X$ is invertible with bounded inverse, then
the constants in Proposition \ref{prop:2.1} can be chosen as $M=1$ and $\omega=0$. In this case $A$ generates a contraction semigroup in $\mathcal H$.
See \cite[Section 1.4]{pazy}.

In particular, if the RK method satisfies Assumption \ref{ass:1.2} and
\begin{equation}\label{eq:2.2}
\sigma(A)\subset \{ z\,:\, \mathrm{Re}\,z\le 0\},
\end{equation}
then the equations \eqref{eq:1.2} (or equivalently \eqref{eq:1.3}), defining the RK approximation of \eqref{eq:1.1} are uniquely solvable for any $k>0$ (apply Lemma \ref{lemma:2.2} with $\mathcal S=k^{-1}\rkA^{-1}$).
The following lemma gives a bound for $\rho_k(T)$ in this specific setting.

\begin{lemma}[Discrete Stability]
\label{lemma:discrete_stability}
Let $\AA$ be a linear, maximally dissipative operator on a Hilbert space $\mathcal{H}$.
For $A$-stable Runge-Kutta methods and  arbitrary $k>0$, we can bound
\begin{equation}
\label{eq:rk_R_is_contraction}
\norm{r(k \,A)}_{\mathcal{H} \to \mathcal{H}}\leq 1,
\end{equation}
and therefore $\rho_k(T)\leq 1$ for all $k$ and $T>0$.
\end{lemma}

\begin{proof}
Let $c(z):=(z+1)/(z-1)$, and note that $c(A)=(A+\id)(A-\id)^{-1}$ is well defined and since
\[
\| (A+\id)x\|^2-\|(A-\id)x\|^2=4\mathrm{Re}\,\langle Ax,x\rangle\le 0 \qquad \forall x\in \dom \AA,
\]
it is clear that $\|c(A)\|_{\mathcal H \to\mathcal H}\le 1.$ Consider now the rational function $q:=r\circ c$. Since $c$ maps $B(0;1)$ bijectively into $\{ z:\mathrm{Re}\,z<0\}$ and $r$ maps the latter set to $B(0;1)$ (this is A-stability), it follows that $q:B(0;1)\to B(0;1)$. Since $\sigma(c(A))\subset \overline{B(0;1)}$ and $c(A)$ is bounded, we can define $q(c(A))$ and show (use a classical result of Von Neumann \cite[Section 4]{von_neumann} or \cite[Chapter XI, Section 154]{riesz_nagy}) that $\|q(c(A))\|_{\mathcal H \to\mathcal H}\le 1$.

Finally, using that $c(c(z))=z$ for all $z$, it follows that $r=q\circ c$. It is then an easy computation to prove that $r(A)=q(c(A))$. (Note that this equality can also be proved using functional calculus.)
\end{proof}

In Propositions \ref{prop:AMP1} and \ref{prop:AMP2}, if $A$ is maximally dissipative, $k_0$ can be chosen arbitrarily. In Lemma \ref{lemma:4.2}, if $\AA$ is maximally dissipative, $k_0$ can be chosen arbitrarily.

\section{Applications}
\label{sect:applications}
In this section, $\Omega$ is a bounded Lipschitz open set in $\R^d$ ($d=2$ or $3$) with boundary $\Gamma$.
%We admit that $\Omega$ has multiple connected components but no cavities ($\Omega$ is simply connected in $d=2$ and the second Betti number of $\Omega$
%is zero if $d=3$). This hypothesis is only imposed for the sake of the application to a scattering problem in Section \ref{sec:9}
%but it does not have any implication on the results of Section~\ref{sect:exotic_transmission}.

We use the usual (fractional) Sobolev spaces $H^s(\Omega)$ for $s\geq 0$ and introduce the space
  $H^1_{\Delta}(\Omega):=\{ u \in H^1(\Omega): \laplace u \in L^2(\Omega) \}$.
  On the boundary $\Gamma$, we also consider Sobolev spaces $H^s(\Gamma)$ and
  their duals $H^{-s}(\Gamma)$. Details can, for example be found in \cite{mclean}.

We will consider the two-sided bounded surjective trace operator $\gamma^\pm:H^1(\R^d\setminus\Gamma)\to H^{1/2}(\Gamma)$
and we will denote $H^{-1/2}(\Gamma)$ for the dual of the trace space. The angled bracket $\langle\,\cdot\,,\,\cdot\,\rangle_\Gamma$ will be used for the $H^{-1/2}(\Gamma)\times H^{1/2}(\Gamma)$ duality pairing and $(\cdot,\cdot)_{\R^d}$ will be used for the inner product in $L^2(\R^d)$ and $\big[L^2(\R^d)\big]^d$.
We will also use the normal trace  $\gamma^\pm_\nu: H(\operatorname{div},\R^d \setminus \Gamma)\to H^{-1/2}(\Gamma)$ and the normal derivative
operators $\partial_\nu^\pm$. Here we make the convention that the normal derivative points out of $\Omega$ for both interior and exterior trace.

We note that the applications in this section are chosen for their simplicity. More complicated applications, also involving
full discretizations by convolution quadrature and boundary elements of systems of time domain boundary integral equations can
be found in \cite{composite_scattering} and \cite{heateq}.

\subsection{Boundary Integral Equations and Convolution Quadrature}
\label{sect:bem_and_cq}
In this section, we give a very brief introduction to boundary integral equations and their discretization using convolution quadrature.
In that way, we can later easily state our methods for both the heat and wave equations in a concise and unified language.
We present the result mostly formally, but note that they can be made rigorous under mild assumptions on the appearing
functions. This theory  can be found in most monographs on boundary element methods, see e.g., \cite{book_sauter_schwab,mclean,steinbach} or \cite{sayas_book}.

For $s \in \C_{+}$, we consider solutions $u \in H^1(\R^d\setminus \Gamma)$ to
the Helmholtz equation
$$-\laplace u - s^2 u =0 \qquad \text{in } \R^d \setminus \Gamma.$$
Using the representation formula, $u$ can be rewritten using only its boundary data:
\begin{align}
  \label{eq:representation_formula}
  u(x)= S(s) \normalDjump{u} - D(s)\tracejump{u},
\end{align}
where the single layer and double layer potentials are given by
\begin{align*}
    \left(S(s) \varphi \right)\left(x\right)&:=\int_{\Gamma}{\Phi(x-y;s) \varphi(y) \;dy}, \\
    \left(D(s) \psi \right)\left(x\right)&:=\int_{\Gamma}{\partial_{n(y)}\Phi(x-y;s) \psi(y) \;dy}.
\end{align*}
We note that both $S(s) \lambda$ and $D(s) \psi$ solve the Helmholtz equation for any given
densities $\lambda \in H^{-1/2}(\Gamma)$ and $\psi \in H^{1/2}(\Gamma)$.

We will need the following four boundary integral operators:
\begin{alignat}{4}
  V(s)&:=\gamma^\pm S(s), &\quad  K(s)&:=\frac{1}{2}(\gamma^+ S(s) + \gamma^- S(s)), \\
  K^t(s)&:=\frac{1}{2}(\partial_\nu^+ D(s) + \partial_\nu^- D(s)), &\quad  W(s)&:=- \partial_\nu D(s).
\end{alignat}

When solving problems in the time domain, we can leverage our knowledge of the Helmholtz equation
using the Laplace transform $\ltrafo$.
For an operator valued analytic function $F$ with $\dom(F) \supset \C_{+}$,
we can then define the convolution operator $F(\partial):=\ltrafo^{-1} \circ \mathrm F \circ \ltrafo$, where $\ltrafo$
is the Laplace transform in the sense of causal distributions.
(Precise definitions can be found in~\cite[Chapter 3]{sayas_book} and~\cite{lubich94}).

Given a Runge-Kutta method, it is then easy to define the convolution quadrature approximation to such operators, as was introduced in~\cite{lubich_ostermann_rk_cq}.
We just replace the Laplace transform by the $Z$-transform and $s$ with the function $\delta/k$, i.e., we define:
\begin{align*}
  F(\dd) g:=\mathscr{Z}^{-1} \left(F\bigg(\frac{\delta(z)}{k}\bigg) \ztrafo{g}\right),
\end{align*}
where $g$ denotes a sequence in the shared domain of $F(s)$ and $k>0$ denotes the stepsize.
The matrix-valued function $F(\frac{\delta(z)}{k})$ is defined using the Riesz-Dunford calculus, but can be
computed in practice by diagonalizing the argument.

\begin{remark}
  We note that our use of the notation $\partial^k$ and $(\partial^k)^{-1}$ is consistent with this definition
  by using the functions $F(s):=s$ and $F(s):=s^{-1}$.
\end{remark}

\subsection{An exotic transmission problem}
\label{sect:exotic_transmission}
In this section we show how to apply Theorems \ref{theorem:3.1}-\ref{theorem:3.3} to a transmission problem in free space associated to the infinitesimal generator of a group of isometries (both $\pm A$ are maximal dissipative) with some exotic transmission conditions which impose partial observation of a trace. In Section \ref{sec:9} we will explain how this problem is related to a boundary integral representation of a scattering problem and how the current results yield the analysis of a fully discrete method for that integral representation.  We keep the presentation brief.
  For more details and exemplary applications we refer to~\cite{sayas_new_analysis}.

Let $Y_h$ be finite dimensional subspace of $H^{1/2}(\Gamma)$ and consider the spaces
\begin{subequations}
\begin{alignat}{6}
\mathbf H(\mathrm{div},\R^d\setminus\Gamma)
	:=&\{\mathbf w\in L^2(\R^d\setminus\Gamma)^d\,:\, 
    \nabla\cdot\mathbf w\in L^2(\R^d\setminus\Gamma)\},\\        
\mathrm V_h :=&\{ v\in H^1(\R^d\setminus\Gamma)\,:\, 
	\tracejump{v} \in Y_h\},\\
\mathbf W_h :=& \{ \mathbf w\in \mathbf H(\mathrm{div},\R^d\setminus\Gamma)
\,:\, \langle \gamma_n^- \mathbf w, \mu_h \rangle_{\Gamma} = 0 \; \quad\forall \mu_h \in Y_h \},\\
\mathbf W_h^0:=& \mathbf W_h\cap \mathbf H(\mathrm{div},\R^d) \\=& 
	\{ w\in \mathbf H(\mathrm{div},\R^d)\,:\, 
		\langle \gamma_n^- \mathbf w, \mu_h \rangle_{\Gamma} = 0 \;\quad \forall \mu_h \in Y_h\}.
\end{alignat}
\end{subequations}
The expression $\tracejump{v}:=\gamma^- v - \gamma^+ v$ denotes the jump of the trace of $v$ across $\Gamma$.
The condition $\tracejump{v} \in Y_h$ is equivalent to
\begin{equation}\label{eq:8.2}
(\nabla\cdot\mathbf w,v)_{\R^d\setminus\Gamma}
	+(\mathbf w,\nabla v)_{\R^d}=0 \qquad \forall \mathbf w\in \mathbf W_h^0.
\end{equation}
We then set
\[
\mathcal X:=L^2(\R^d \setminus \Gamma)\times L^2(\R^d\setminus\Gamma)^d,
	\qquad
\mathcal V:=\mathrm V_h\times \mathbf W_h,
	\qquad
\mathcal M:=H^{-1/2}(\Gamma).
\]
In $\mathcal X$ we use the natural inner product, in $\mathcal V$ we use the norm of $H^1(\R^d\setminus\Gamma)\times \mathbf H(\mathrm{div},\R^d\setminus\Gamma)$, and in $\mathcal M$ we use the usual norm. We will define $A_\star:\dom A_\star=\mathcal V\to \mathcal X$ and $B:\mathcal V \to \mathcal M$ by 
\[
A_\star(v,\mathbf w):=(\nabla\cdot\mathbf w,\nabla v),
\qquad
B(v,\mathbf w):=\gamma_n^- \mathbf w -\gamma_n^+ \mathbf w,
\]
understanding that $A_\star$ can also be extended to $H^1(\R^d\setminus\Gamma)\times \mathbf H(\mathrm{div},\R^d\setminus\Gamma)$. As we did in Assumption \ref{ass:1.1}, we consider $\operatorname{dom} A=\ker B =\mathrm V_h\times \mathbf W_h^0$ and define $A$ as the restriction of $A_\star$ to this subset. 
 
\begin{proposition}\label{prop:8.0}
The operators $\pm A$ are maximal dissipative.
\end{proposition}

\begin{proof}
The identity \eqref{eq:8.2} shows that $\langle A(v,\mathbf w),(v,\mathbf w)\rangle_{\mathcal X}=0$ for all $(v,\mathbf w)\in \mathrm V_h \times \mathbf W_h^0.$ Given $(f,\mathbf f)\in \mathcal X$, solving the coercive problem
\[
v\in \mathrm V_h, \qquad (\nabla v,\nabla w)_{\R^d}+(v,w)_{\R^d}
	=(f,w)_{\R^d}-(\mathbf f,\nabla w)_{\R^d} \quad \forall w\in \mathrm V_h,
\]
and defining $\mathbf w=\nabla v+\mathbf f$, we have a pair $(v,\mathbf w)\in \mathrm V_h \times \mathbf W_h^0$ such that $(v,\mathbf w)-A(v,\mathbf w)=(f,\mathbf f)$ and thus $A$ is maximal dissipative. 
The proof of the maximal dissipativity of $-A$ is similar. (Note that this is a particular case of what appears in \cite{sayas_new_analysis}.)
\end{proof}

We consider the standard problem \eqref{eq:1.1} with vanishing initial conditions and data $F=0$ and $\Xi=g:[0,\infty)\to  L^2(\Gamma)$, namely, we look for $(v_h,\mathbf w_h):[0,\infty)\to \operatorname{dom} A_\star$ such that
\begin{subequations}\label{eq:8.3}
\begin{alignat}{6}
& (\dot v_h(t),\dot{\mathbf w}_h(t))= 
	(\nabla\cdot \mathbf w_h(t),\nabla v_h(t))\quad && \forall t>0,\\
& \big\langle\gamma_\nu^+ \mathbf w_h(t)-\gamma_\nu^- \mathbf w^h(t),\mu\big\rangle_{\Gamma}=\left<g(t),\mu\right>_{\Gamma} \quad
	&& \forall \mu \in Y_h, \quad\forall t>0,\\
& (v_h(0),\mathbf w_h(0))=(0,\mathbf 0).
\end{alignat}
\end{subequations}
Uniqueness of the solution to \eqref{eq:8.3} follows from Proposition \ref{prop:8.0}. We will handle existence of a solution below.
The quantities of interest are $u_h:=\partial^{-1}v_h$ and its Dirichlet trace $\psi_h:=\tracejump{ u_h}:\; [0,\infty) \to Y_h$.

\begin{proposition}\label{prop:8.2}
There exists a linear bounded  right inverse of $B$, $\lifting :\mathcal M \to \operatorname{dom} A_\star$ such that $\mathrm{range}\,\lifting\subset \ker (I-A_\star)$. The norm of $\lifting$ is independent of the space $Y_h$.
\end{proposition}

\begin{proof}
Given $\xi\in \mathcal M=H^{-1/2}(\Gamma)$, we solve the coercive problem 
  \begin{alignat}{6}
    \label{eq:8.4}
& v\in \mathrm V_h, \quad 
& (\nabla v,\nabla w)_{\R^d \setminus \Gamma}+(v,w)_{\R^d \setminus \Gamma}= \langle \xi, \gamma^+ {w} \rangle_{\Gamma} \quad \forall w\in \mathrm V_h,
\end{alignat}
and then define $\mathbf w:=\nabla v$.

This problem is equivalent to (note \eqref{eq:8.2})  
\begin{equation}
(v,\mathbf w)\in \operatorname{dom} A_\star,
	\qquad
(v,\mathbf w)=A_\star(v,\mathbf w),
	\qquad
B(v,\mathbf w)=\xi.
\end{equation} 
Since $\abs{\langle \xi, \gamma^+ {w} \rangle_{\Gamma}} \leq \norm{\xi}_{H^{-1/2}(\Gamma)} \norm{w}_{H^1(\R^d \setminus \Gamma)} $
it follows that the norm of the solution operator for \eqref{eq:8.4} is independent of the space $Y_h$.  
\end{proof}

\begin{proposition}\label{prop:8.1}
  The lifting $\lifting$ from Proposition~\ref{prop:8.2} is a bounded linear
    map $L^2(\Gamma) \to \mathcal X_{1/2}:=[\mathcal X,\operatorname{dom} A]_{1/2}$
    with
    \begin{align*}
      \norm{\lifting \lambda}_{\XX_{1/2}} \leq C \norm{\lambda}_{L^2(\Gamma)}.
    \end{align*}
    $C$ depends only of $\Omega$.
\end{proposition}
\begin{proof}
  
    We will need spaces encoding homogeneous normal traces:
    \begin{align*}
    \mathbf H_0(\mathrm{div},\Omega)&:=\big\{ \bs w \in \mathbf H(\mathrm{div},\Omega): \gamma_{\nu}^- \bs w =0\big\}, \\     
    \mathbf H_0(\mathrm{div},\R^d\setminus\overline{\Omega}))&:=\big\{ \bs w \in \mathbf H(\mathrm{div}, \R^d\setminus\overline{\Omega})): \gamma_{\nu}^+ \bs w =0\big\}.
    \end{align*}  
    By applying Theorem~\ref{thm:lifting_of_neumann}
    to the exterior, and setting $\widetilde{\mathbf{w}} =0 $ inside,
    we can construct a function    
    $
    \widetilde{\mathbf{w}} \in \mathbf H(\mathrm{div},\R^d\setminus\Gamma)
    $,  
    satisfying $\normaltracejump{\widetilde{\mathbf{w}}} = \lambda$
    and
    $$
    \norm{\widetilde{\mathbf{w}}}_{[L^2(\Omega),  \mathbf H_0(\mathrm{div},\R^d\setminus\overline{\Omega})]_{1/2}}
    \lesssim \norm{\lambda}_{L^2(\Gamma)}.
    $$

    Since (up to identifying the product space with the spaces on $\R^d \setminus \Gamma$), it holds that
    $$
    \mathbf H_0(\mathrm{div},{\Omega})
    \times \mathbf H_0(\mathrm{div},\R^d\setminus\overline{\Omega})
    \subseteq {\bs W}_h^0.
    $$
    The product of interpolation spaces equals the interpolation of product spaces (cf. \cite[Sect. 1.18.1]{triebel95}) we can therefore also estimate:
    $$
    \norm{ (0,\widetilde{\mathbf{w}})}_{\mathcal{X}_{1/2}}
    \lesssim 
    \norm{\widetilde{\mathbf{w}}}_{[L^2(\Omega),  \bs W_h^0]_{1/2}} \lesssim \norm{\lambda}_{L^2(\Gamma)}.
    $$

  If we consider $(v,\mathbf{w}) := \lifting\lambda$, then $(v, \mathbf{w} - \widetilde{\mathbf{w}}) \in \dom(\AA)$ by
  construction of the lifting. Thus we have
  \begin{align*}
    \norm{(v,\mathbf{w})}_{\mathcal{X}_{1/2}}
    &\leq  \norm{(v,\mathbf{w} - \widetilde{\mathbf{w}})}_{\mathcal{X}_{1/2}}
    + \norm{ (0, \widetilde{\mathbf{w}})}_{\mathcal{X}_{1/2}} \\
    &\leq  \left(\norm{v}_{H^1(\R^d \setminus \Gamma)} + \norm{\mathbf{w} - \widetilde{\mathbf{w}}}_{H(\operatorname{div},\R^d\setminus \Gamma)}\right)
    + \norm{ (0,\widetilde{\mathbf{w}})}_{\mathcal{X}_{1/2}}.                                            
  \end{align*}  
\end{proof}

\begin{proposition}
If $g\in \mathcal C^2([0,\infty);H^{-1/2}(\Gamma))$ satisfies $g(0)=\dot g(0)=0$, then \eqref{eq:8.3} has a unique strong solution. 
\end{proposition}

\begin{proof}
Thanks to Propositions \ref{prop:8.0} and \ref{prop:8.2}, this problem fits in the abstract framework described in \cite{sayas_new_analysis}, which proves existence and uniqueness of solution to \eqref{eq:8.3}.
\end{proof}

Propositions \ref{prop:8.0} -- \ref{prop:8.2} have some consequences. First of all, Assumption \ref{ass:1.1} holds. Secondly,
assuming $g(t) \in L^2(\Gamma)$, any solution to \eqref{eq:1.1} with the above data ($F=0$, $\Xi=g$) is in $\mathcal X_{1/2}$,
and therefore, solutions to \eqref{eq:8.3} take values in $\mathcal X_{1/2}$ as well.
Finally, if $g\in \mathcal C^s([0,\infty];L^2(\Gamma))$ then $\lifting g\in \mathcal C^s([0,\infty];\mathcal X_{1/2})$.

We also need a regularity result that allows us to bound time derivatives of the solution in terms of the data. The continuity condition for the $(s+2)$-th derivative of $g$ in Proposition \ref{prop:8.5} can be relaxed to local integrability, but then the norms on the right-hand side of \eqref{eq:8.7} have to be modified. 

\begin{proposition}\label{prop:8.5}
If $g\in \mathcal C^{s+2}([0,\infty);L^2(\Gamma))$ satisfies $g^{(\ell)}(0)=0$ for $\ell\le s+1$, then the unique solution to \eqref{eq:8.3} satisfies
\begin{itemize}
\item[(a)] $(v_h,\mathbf w_h)\in \mathcal C^{s+1}([0,\infty);\mathcal X)$,
\item[(b)]
  \label{it:prop:8.5:b}
  $(v_h,\mathbf w_h)\in \mathcal C^{s}([0,\infty);\mathcal V)$ and  $(v_h,\mathbf w_h)\in \mathcal C^{s}([0,\infty);\mathcal X_{1/2})$,
\item[(c)] for all $\ell\le s$, there exists $C$, independent of the choice of $Y_h$ such that  for all $t\ge 0$
\begin{equation}\label{eq:8.7}
\| (v_h^{(\ell)}(t),\mathbf w_h^{(\ell)}(t))\|_{\mathcal X_{1/2}}
	\le C\,t\,\sum_{j=\ell}^{\ell+2} \max_{\tau\le t} \| g^{(j)}(\tau)\|_{L^2(\Gamma)}.
\end{equation}
\end{itemize}
\end{proposition}

\begin{proof}
  This result follows from \cite[Theorem 3.1]{sayas_new_analysis}.
  To see~\ref{it:prop:8.5:b}, we note that $(v_h,\mathbf{w}_h)$ is constructed
    by writing
    $$
    (v_h(t),\mathbf{w}_h(t))= (v^0_h(t),\mathbf{w}^0_h(t)) + \lifting g(t),
    $$
    with $(v^0_h(t),\mathbf{w}^0_h(t)) \in \dom(\AA)$. The statement then follows from
    Proposition~\ref{prop:8.1}.
\end{proof}

We now consider the RK approximation of \eqref{eq:8.3} in a finite time interval $[0,T]$, providing pairs of stage-values $(V^k_{h,n},\mathbf W^k_{h,n})\in \mathcal X^m$ and step approximations $(v^k_{h,n},\mathbf w^k_{h,n})\in \mathcal X$. We then define
\begin{equation}
\{ U^k_{h,n}\}=(\partial^k)^{-1}\{ V^k_{h,n}\},
	\qquad
	u^k_{h,n}=r(\infty)u^k_{h,n}+\rkb^\top \rkA^{-1}U^k_{h,n}, \quad n\ge 0
\end{equation}
with $u^k_{h,0}=0$ (see Lemma \ref{lemma:2.6}) and $\psi^k_{h,n}:=\tracejump{u^k_{h,n}}$ . 

\begin{proposition}
  \label{prop:convergence_rates_dirichlet}
  For sufficiently smooth  $g$, with RK approximations using a method satisfying Assumption \ref{ass:1.2}, and with $\alpha$ given by \eqref{eq:6.0}, for $nk\le T$ we have the estimates 
  \begin{align}
    \label{eq:8.8}
    \| u_h(t_n)-u^k_{h,n}\|_{L^2(\R^d\setminus\Gamma)}
	&\leq C T^2 k^{\min\{q+3/2+\alpha,p\}} \sum_{\ell=q}^{p+3} 
    \max_{t\le T} \|g^{(\ell)}(t)\|_{L^2(\Gamma)},
  \end{align}
  and
  \begin{multline}
    \| u_h(t_n)-u^k_{h,n}\|_{H^1(\R^d\setminus\Gamma)} +  \| \psi_h(t_n)-\psi^k_{h,n}\|_{H^{1/2}(\Gamma)}\\
    \leq C  T^2 k^{\min\{q+1/2+\alpha,p\}}  \sum_{\ell=q}^{p+3} 
    \max_{t\le T} \|g^{(\ell)}(t)\|_{L^2(\Gamma)}. \label{eq:8.9}   
  \end{multline}

  The constants depend on $\Gamma$ and the Runge-Kutta method, but do not depend on $T$ or on the choice of $Y_h$.
\end{proposition}

\begin{proof}
We will use Theorems \ref{theorem:3.1} and \ref{theorem:3.3} as well as Propositions \ref{prop:AMP1} and \ref{prop:AMP2}. We note that $\rho_k(T)\le 1$ by Lemma \ref{lemma:discrete_stability} and Proposition \ref{prop:8.0}. Also, with the $\lifting$ operator of Proposition \ref{prop:8.2} , we have
\begin{equation}\label{eq:8.11}
\|\lifting\Xi^{(\ell)}\|_{\mathcal X_{1/2}}\le C \|g^{(\ell)}\|_{L^2(\Gamma)},
\end{equation}
with $C$ independent of $Y_h$. The bound~\eqref{eq:8.8}
follows from Theorem \ref{theorem:3.1}, using \eqref{eq:8.7} and \eqref{eq:8.11} to estimate the right-hand side. The bound
\begin{alignat}{6}
\nonumber
\| \nabla u_h(t_n)-\nabla u^k_{h,n}\|_{L^2(\R^d\setminus\Gamma)}
& =\|\mathbf w_h(t_n)-\mathbf w^k_{h,n}\|_{L^2(\R^d \setminus \Gamma)}\\
& \le C T^2 k^{\min\{q+1/2+\alpha,p\}} 
\sum_{\ell=q+1}^{p+3} \max_{t\le T} \| g^{(\ell)}\|_{H^{1/2}(\Gamma)}
\label{eq:8.13}
\end{alignat}
follows from Propositions \ref{prop:AMP1} and \ref{prop:AMP2}, using \eqref{eq:8.7} for the estimate in terms of the data.
The $H^1(\R^d\setminus\Gamma)$ estimate \eqref{eq:8.9} is then a direct consequence of  \eqref{eq:8.8} and \eqref{eq:8.13},
the estimate for $\psi_h-\psi_h^k$ follows from the standard trace theorem.
\end{proof}

\subsection{Scattering}\label{sec:9}
We stay in the geometric setting of the previous section. Assume that $\mathbf d\in \R^d$ is a unit vector (direction of propagation) and that $c\in \R$ is such that $\Omega\subset \{\mathbf x\in \R^d\,:\, \mathbf x\cdot\mathbf d> c\}$. Let $\phi:\R\to \R$ be a function such that $\phi(r)=0$ for all $r\ge c.$ The incident wave $u^{\mathrm{inc}}(\mathbf x,t):=\phi(\mathbf x\cdot\mathbf d-t)$, propagates in the direction $\mathbf d$ at unit speed and has not reached the scatterer given by $\Omega$ at time $t=0$. The data for our problem will be the function $g:[0,T]\to L^2(\Gamma)$ given by
$g(t):=-\partial_\nu u^{inc}(\cdot,t)$.

The scattering problem by a sound-hard obstacle occupying the domain $\Omega$ looks for the scattered field $u:[0,T]\to H^1(\R^d\setminus\overline\Omega)$ satisfying
\[
\ddot u(t)=\Delta u(t), \qquad u(0)=\dot u(0)=0, \qquad \partial^+_\nu u(t)=g(t),
\]
so that $\partial^+_\nu (u+u^{\mathrm{inc}})=0$. (Note that we can take the trace of the normal derivative of the incident wave,
since it is locally smooth.)

A direct formulation for solving this problem is equivalent to an extension of $u$ to the interior domain by zero. This means we solve
\begin{equation}\label{eq:9.1}
\ddot u(t)=\Delta u(t) \,\mbox{ in $\R^d\setminus\Gamma$}, \quad u(0)=\dot u(0)=0, \quad \normalDjump{u(t)}=g(t), \quad  \partial_\nu^- u(t)=0.
\end{equation}

With some additional hypotheses on the growth of $g$
(which is needed to have a well-defined distributional Laplace transform), we can represent the solution to \eqref{eq:9.1} as $u=\mathrm S(\partial)g - D(\partial) \psi$,
where $\psi:=\tracejump{u}$. Note that, to be precise with the use of weak distributional definitions, all functions have to be extended by zero to $t<0$ (we say that they are causal) and the time interval is extended to infinity.

Taking the trace in this representation formula,
the solution of \eqref{eq:9.1} can be found by solving an equation for $\psi$ and then postprocessing with the potential operators:
\begin{equation}\label{eq:9.4}
{\mathrm W}(\partial)\psi=(1/2 - K^t(\partial)) g, \qquad u=\mathrm S(\partial) g - \mathrm{D}(\partial) \psi,
\end{equation}
and we still have that $\psi=\tracejump{u}$.
We can equivalently write \eqref{eq:9.1} and the equivalent \eqref{eq:9.4} by using the variables $v:=\dot u$ and $\mathbf w:=\nabla u$.
We note that $u=\partial^{-1} v$ and $\psi=\partial^{-1} \tracejump{v}$. Here, $(v,\mathbf w)$ solve (we restrict $t$ to the interval $[0,T]$ again)
\[
\dot v(t)=\nabla\cdot\mathbf w(t), 
\quad 
\dot{\mathbf w}(t)=\nabla v(t),
\quad
\normaltracejump{\mathbf{w}(t)}=g(t),
\quad
v(0)=0, \quad \mathbf w(0)=\mathbf 0,
\]
that is, \eqref{eq:8.3} with $Y_h=H^{1/2}(\Gamma)$. For the discretization,
we consider a finite dimensional space $Y_h$ and the Galerkin approximation to \eqref{eq:9.4}, so that we look for
$\psi_h:\,\R\to X_h$ causal such that
\begin{equation}\label{eq:9.5}
\langle  \mathrm W(\partial)\psi_h, \mu \rangle_\Gamma= \langle (1/2 - K^t(\partial)) g, \mu \rangle_\Gamma
	\quad\forall \mu \in Y_h,
\qquad
u_h:=\mathrm S(\partial)g - D(\partial) \psi_h.
\end{equation}
The functions $v_h:=\dot{u}_h$ and $\mathbf w_h:=\nabla u_h$ satisfy \eqref{eq:8.3}. The difference between the solutions of \eqref{eq:9.1} and \eqref{eq:9.5}
can be studied by comparing the solutions to \eqref{eq:8.3} when $Y_h=H^{1/2}(\Gamma)$ and when $Y_h$ is a finite dimensional space, see~\cite{sayas_new_analysis}
for details. For our purposes, it is sufficient to note that we get quasi-optimal estimates for the discretization in space.

Discretization in time is performed by applying convolution quadrature to~\eqref{eq:9.5}. The fully discrete solution reads
  \begin{equation}
    \label{eq:9.6}
  \langle  \mathrm W(\partial^k)\Psi_h, \mu \rangle_\Gamma= \langle (1/2 - K^t(\partial^k)) g, \mu \rangle_\Gamma
  \;\forall \mu \in Y_h,\; \;
  U_h:=\mathrm S(\partial^k)g - D(\partial^k) \Psi_h.
\end{equation}

The approximations $\psi^k_h$ and $u^k_h$ are then computed by the usual post-processing, i.e.
\begin{alignat*}{3}
  \psi^k_{h,0}&:=0, \qquad &\psi^k_{h,n+1}&=r(\infty) \psi^k_{h,n} + \rkb^T \rkA^{-1} \Psi^k_{h,n}, \\
  u^k_{h,0}&:=0, \qquad &u^k_{h,n+1}&=r(\infty) u^k_{h,n} + \rkb^T \rkA^{-1} U^k_{h,n}.
\end{alignat*}

\begin{lemma}
  The sequences $u^k_h$ and $\psi_h^k$ computed via~\eqref{eq:9.6} coincide with the Runge-Kutta
  approximations to \eqref{eq:8.3} and their traces respectively.
\end{lemma}
\begin{proof}
  The details of the computation can be found in the appendix of~\cite{cq_schroedinger}. The basic idea
  is to take the Z-transform and show that both approaches solve the matrix-valued Helmholtz problem~\eqref{eq:3.3}.
\end{proof}

This gives the following immediate corollary, representing an apriori bound for the fully discrete method:
\begin{corollary}
  Let the assumptions of Proposition~\ref{prop:convergence_rates_dirichlet} hold. Then
  for $u_h$ and $\psi_h$, approximated using convolution quadrature, we can estimate:
  \begin{multline}
    \label{eq:8.19}
    \| u_h(t_n)-u^k_{h,n}\|_{1(\R^d\setminus\Gamma)} +  \| \psi_h(t_n)-\psi^k_{h,n}\|_{H^{1/2}(\Gamma)}\\
    \leq C  (1+T^2) k^{\min\{q+1/2+\alpha,p\}}  \sum_{\ell=q}^{p+3} 
    \max_{t\le T} \|g^{(\ell)}(t)\|_{L^2(\Gamma)}.  
  \end{multline}

  The constants depend on $\Gamma$ and the Runge-Kutta method, but do not depend on $T$ or on the choice of $Y_h$.
\end{corollary}

\begin{remark}
  There is another approach for analyzing convolution quadrature methods, which is based on estimates in the Laplace domain.
  It can be shown that the Neumann-to-Dirichlet map, realized by the boundary integral equations~\eqref{eq:9.6},
  satisfies a bound of the form
  $$ \norm{W(s)^{-1}(1/2 - K^t(s)) \widehat{g}}_{H^{1/2}(\Gamma)}\lesssim \frac{\abs{s}}{\Re(s)} \norm{g}_{H^{-1/2}(\Gamma)},$$
  see~\cite[Appendix 2]{laliena_sayas}.
  Applying the abstract theory of~\cite{BanLM} then implies convergence rate
  $\min(q+1,p)$ for the boundary data $\psi_h$.
  Modifying their proof, one can also get for $g(t)\in L^2(\Gamma)$ that
    $$
    \norm{W(s)^{-1}(1/2 - K^t(s)) \widehat{g}}_{H^{1/2}(\Gamma)}\lesssim \frac{\abs{s}^{1/2}}{\Re(s)} \norm{g}_{L^2(\Gamma)},
    $$
  which would yield the same convergence rate as Corollary~(\ref{eq:8.19}), but without insight into
  the dependence on the end-time $T$.
\end{remark}

\subsection{Numerical example}
\label{sect:scattering_numerics}
We solve~\eqref{eq:9.6} on a ``hollow square'', as depicted in Figure~\ref{fig:snaphshots_simulation}.
The geometry was chosen to be non-convex and not simply connected, in order to
test if the rate observed is a general result, or if our estimates might prove sharp
in some situation.

We prescribe the exact solution as a traveling wave,
given by

\begin{align*}
  u(\mathbf x, t)&:=\phi(\mathbf x \cdot \mathbf{d} - t), \\
  \phi(s)&:=\cos(\pi \,s/2) \, \exp(-4(s_0-s)^2).
\end{align*}
$s_0:=4$ is chosen so that $\phi(0)$ is sufficiently small in the domain. We set $\mathbf{d}:=[\frac{\sqrt{2}}{2},\frac{\sqrt{2}}{2}]^\top$
and solve up to an end time of $T=12$. An approximation of the $H^{1/2}$-error
is computed via
$$\left\langle \mathrm{W}(1) \left( \psi^{k}_{h,n}- \Pi_{L^2} \psi(t_n)\right), \psi^{k}_{h,n}- \Pi_{L^2} \psi(t_n) \right\rangle_{\Gamma},$$
i.e., we compare to the $L^2$-projection of the exact solution. Since we are interested in the convergence rate
with respect to the timestep size $k$, we consider a fixed, but sufficiently fine mesh.

We used 3 and 5 stage Radau~IIA methods, with orders $(q,p)$ of $(3,5)$ and $(5,9)$, respectively
(see~\cite{hairer_wanner_2} for their definition).
Our theory predicts convergence rates  of $4.5$ and $6.5$.
In Figure~\ref{fig:convergence_neumann}, we observe a rate that
is closer to $5$ and $8$. This means that (just like the standard Laplace-domain estimates) our
  estimates do not appear to be sharp in this case. Further investigations into the cause of this phenomenon are required.
 Results trying to explain this phenomenon, initially prompted by the work on this article,
can be found in~\cite{superconvergence} but with a different model problem.

\begin{figure}[htb]
  \centering
  \newcommand{\mywidth}{5cm}
  \begin{subfigure}{0.3\textwidth}
    \includegraphics[width=\mywidth]{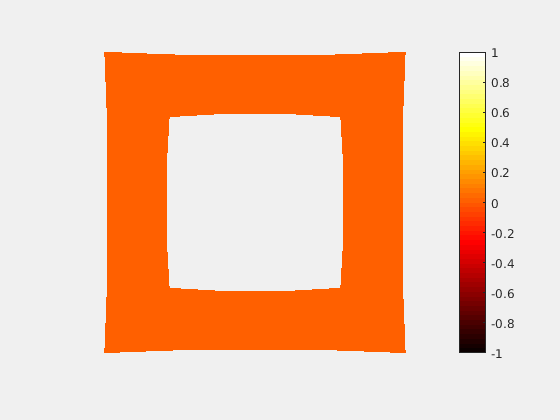}%
    \caption{$t=0$}
  \end{subfigure}%
  \begin{subfigure}{0.3\textwidth}
    \includegraphics[width=\mywidth]{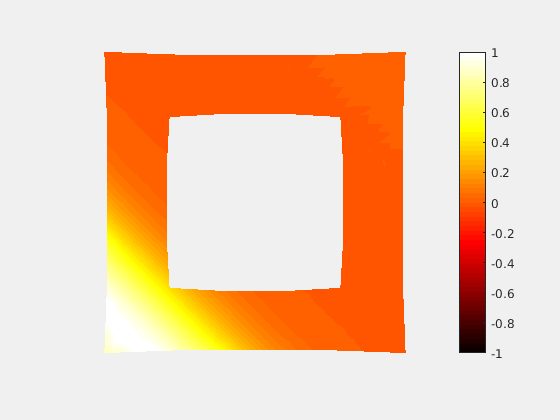}%
    \caption{$t=2.85$}
  \end{subfigure}%
    \begin{subfigure}{0.3\textwidth}
    \includegraphics[width=\mywidth]{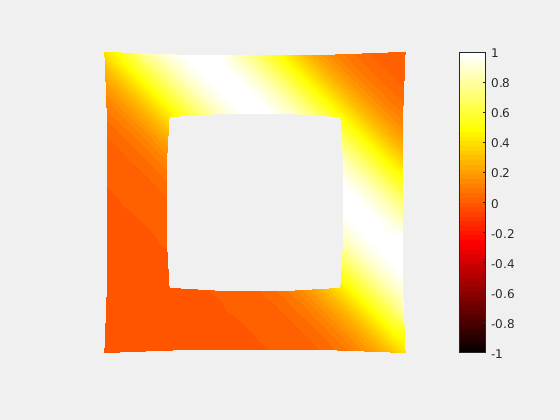}
    \caption{$t=4.45$}
  \end{subfigure}
  \begin{subfigure}{0.3\textwidth}
    \includegraphics[width=\mywidth]{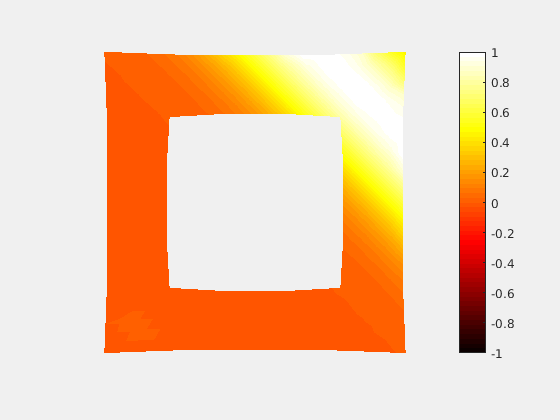}%
    \caption{$t=5.0$}
    \end{subfigure}
    \begin{subfigure}{0.3\textwidth}
    \includegraphics[width=\mywidth]{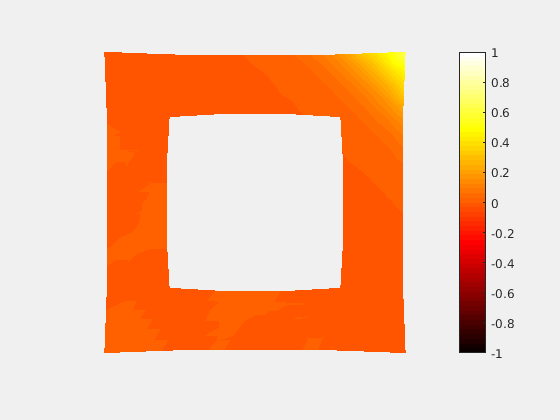}%
    \caption{$t=5.6$}
  \end{subfigure}
  \begin{subfigure}{0.3\textwidth}
    \includegraphics[width=\mywidth]{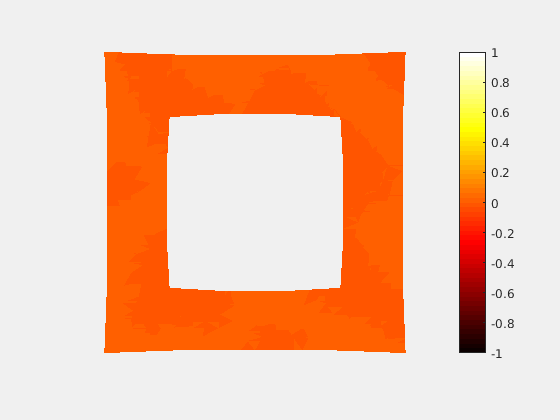}
    \caption{$t=12.0$}
  \end{subfigure}  
  \caption{Snapshots of the simulation at $t=2.85$, $t=4.45$, $t=5.0$, $t=5.6$, $t=12$}
  \label{fig:snaphshots_simulation}
\end{figure}

\begin{figure}[htb]
  \centering
    \includeTikzOrEps{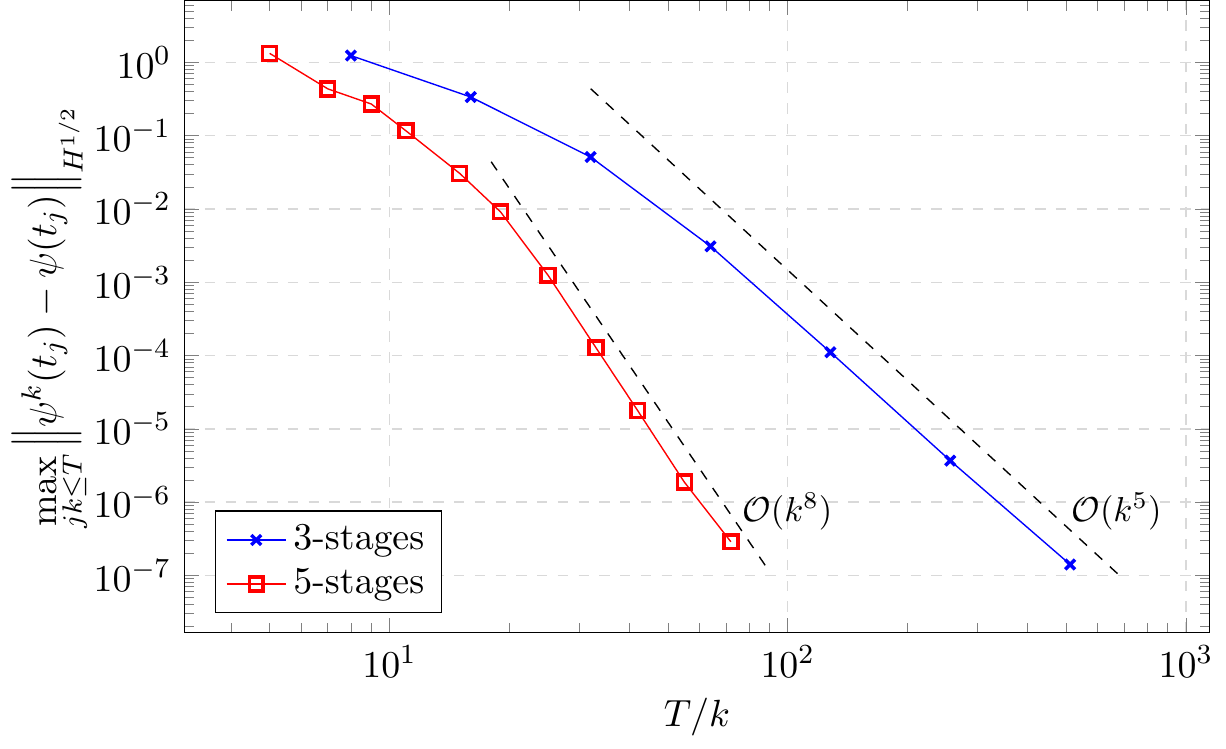}    
    \caption{Performance of Radau~IIA methods for the wave equation,  cf. Section~\ref{sect:scattering_numerics}}
    \label{fig:convergence_neumann}
\end{figure}

\subsection{The Heat equation}
\label{sect:heat}
In this section, as an example where our estimates turn out to be sharp, we consider a heat conduction problem
and will apply the Theorem~\ref{theorem:3.3} to get convergence of the boundary trace. The physical situation
is a body $\Omega \subset \R^d$ that is held at a given temperature distribution and radiates heat into a medium $\Omega^+:=\R^d \setminus \Omega$.
We make the simplifying assumption that at $t=0$ the temperature is $0$.
Since the problem is posed on an unbounded domain, it is a good candidate for boundary integral equations, while being simple enough
to showcase our more general results.
We only briefly give the mathematical setting. More details and a more involved physical
example can be found in~\cite{heateq}.
The setting is as follows: find $u: \R_+ \to H^1_{\Delta}(\Omega^+)$ such that
\begin{subequations}
  \label{eq:heat_eqn}
  \begin{alignat}{3}
    \dot{u} &= \laplace u  &\quad  & \text{in } \R^d \setminus \overline{\Omega},  \label{eq:heat_eqn:pde}\\
    u(t)|_{\Gamma}&= g(t) & \quad &\text{on } \Gamma:=\partial \Omega,  \label{eq:heat_eqn:bc} \\
    u(0)&= 0 &\quad  & \text{in } \R^d \setminus \overline{\Omega}.  \label{eq:heat_eqn:ic}
  \end{alignat}
\end{subequations}
It is well known that $\laplace$ with homogeneous Dirichlet boundary conditions
generates an analytic semigroup (see e.g.~\cite[Section 7.2]{pazy})
on $L^2(\R^d \setminus \overline{\Omega})$. The rest of our assumptions are also easily checked. We summarize:
\begin{enumerate}[(i)]
\item $\dom(\AAstar)=\{u \in H^1(\R^d \setminus \overline{\Omega}): \laplace u \in L^2(\R^d \setminus \overline{\Omega}) \}$,
\item $B: H^1(\R^d \setminus \overline{\Omega}) \to H^{1/2}(\Gamma)=:\mathcal{M}, \; Bv := \gamma^+ v$ (using the standard trace operator).
\end{enumerate}

In order to derive the boundary integral formulation, we take the Laplace transform of~\eqref{eq:heat_eqn:pde}, giving
for $\kappa:=\sqrt{s}$:
\begin{align*}
   - \laplace \widehat{u}(s) + \kappa^2 \widehat{u}(s) &= 0,
\end{align*}
which is Helmholtz's equation for a complex wave number $\kappa$. We make an ansatz of the
form $\widehat{u}=S(\kappa) \widehat{\lambda}$ for some unknown density $\widehat{\lambda}$, which can be determined by
applying the trace operator, giving the equation $V(\kappa)\widehat{\lambda}=\ltrafo(g)$.

Transforming back, and using the definition $V_{\kappa}(s):=V(\sqrt{s})$, we get the formulation:
\begin{align*}
  \left[V_{\kappa}(\partial) \lambda\right](t) &= g(t) \qquad \forall t>0.
\end{align*}
The solution $u$ can then be recovered by computing $u=S_{\kappa}(\partial)$, where $S_\kappa(s):=S(\sqrt{s})$.

The discrete version of this is then given by solving
\begin{align}
  \label{eq:heat_eq_bem}
  V_{\kappa}(\partial^k) \Lambda^k &= g.
\end{align}

It can be shown that plugging the discrete solution into the representation formula
$U^k:=S_{\kappa}(\partial^k) \Lambda^k$ gives back the Runge-Kutta approximation of~\eqref{eq:heat_eqn}.
The approximations at the endpoints $t_n=n\,k$, denoted by $\lambda^k$ and $u^k$ respectively can be computed by the usual post-processing.
We refer to the appendix of~\cite{cq_schroedinger} for an analogous computation in the context of the Schr\"odinger equation, which easily transfers to our situation.
For simplicity, we do not consider any discretization in space. A Galerkin approach could easily be included into the analysis,
analogously to Section~\ref{sect:exotic_transmission}.

We need the following analog of Proposition~\ref{prop:8.1}:
\begin{proposition}
  \label{prop:interpolation_space_heat}
  For $\mu \in [0,1/4]$, we have $\dom(\AAstar) \subseteq [L^2(\R^d \setminus \overline{\Omega}), \dom(\AA)]_{\mu,\infty}$.
\end{proposition}
\begin{proof}
  It is easy to see that $H^2_0(\R^d \setminus \overline{\Omega}) \subseteq \dom(\AA)$.
  
    Using the Besov spaces introduced in Appendix~\ref{appendix:interpolation_of_sobolev},
    we can write, if $\mu\leq 1/4$:
  \begin{align*}
    H^1(\R^d \setminus \overline{\Omega})
    &\subseteq B^{2\mu}_{2,1}(\R^d \setminus \overline{\Omega}) 
     \stackrel{Thm~\ref{thm:interpolation_spaces_coincide}}{\subseteq} 
    \widetilde{B}^{2\mu}_{2,\infty}(\R^d \setminus \overline{\Omega}) 
    =[L^2(\R^d \setminus \overline{\Omega}), H_0^1(\R^d \setminus \overline{\Omega})]_{2\mu,\infty} \\
    &=[L^2(\R^d \setminus \overline{\Omega}), H_0^2(\R^d \setminus \overline{\Omega})]_{\mu,\infty},
  \end{align*}
  where in the last step, we used~\cite[Theorem B.9]{mclean}.
\end{proof}

The convergence of our numerical method can then be analyzed quite easily using
Proposition~\ref{prop:AMP2} and Theorem~\ref{theorem:3.3}.
\begin{theorem}
  \label{thm:convergence_heat}
  Let $g \in \mathcal{C}^{p+3}([0,T], H^{1/2}(\Gamma))$ with $g^{(j)}(0)=0$ for $j=0,\dots p+2$.
  Let $p$ and $q$ denote the classical and stage order of the Runge-Kutta method used.
  Then the following estimate holds for the post-processed approximation:
  \begin{align}
    \label{eq:10.3}
    \norm{u^{k}(t_n) - u(t_n)}_{L^2(\R^d\setminus \overline{\Omega})}&\leq C (1+ T^2)  k^{\min(q+\alpha+1/4,p)}\sum_{\ell=q+1}^{p+2}{\norm{g^{(\ell)}}_{H^{1/2}(\Gamma)}}.
  \end{align}
  Assume that the Runge-Kutta method used for discretization is stiffly accurate. Then the following estimates hold for
  the $H^1$-norm:
  \begin{align}
    \label{eq:10.4}
    \norm{u^{k}(t_n) - u(t_n)}_{H^1(\R^d\setminus \overline{\Omega})}&\leq C (1+T^2)  k^{r_1}\sum_{\ell=q+1}^{p+3}{\norm{g^{(\ell)}}_{H^{1/2}(\Gamma)}},
  \end{align}
  with
  \begin{align*}
    r_1:=\begin{cases}
      q+ \alpha - 1/4 & \text{ for } q< p-1, \\
      q -1/4 & \text{ for } q= p-1 \text{ and $\alpha=0$}, \\
      q+ 5/8  & \text{ for } q= p-1 \text{ and $\alpha=1$},\\   
      q + \frac{\alpha-1}{2} & \text{ for } q=p.
      \end{cases}
  \end{align*}

  And for the density, we get:
  \begin{align}
    \label{eq:10.5}
    \norm{\lambda^{k}(t_n) - \lambda(t_n)}_{H^{-1/2}(\Gamma)}&\leq C 1+T^2  k^{r_{\lambda}}\sum_{\ell=q}^{p+1}{\norm{g^{(\ell)}}_{H^{1/2}(\Gamma)}},
  \end{align}
  where the rate $r_{\lambda}$ is given by:
  \begin{align*}
    r_{\lambda}:=\begin{cases}
      q+ \alpha - 1/2 & \text{ for } q< p-1, \\
      q - 1/2 & \text{ for } q= p-1 \text{ and $\alpha=0$}, \\
      q+ \frac{7}{16}  & \text{ for } q= p-1 \text{ and $\alpha=1$},\\   
      q + \frac{3}{4}  (\alpha-1)  & \text{ for } q=p.
      \end{cases}
  \end{align*}
\end{theorem}
\begin{proof}
  We first note that we can control the derivatives $u^{(\ell)}$ by the data. This can be done completely analogous
  to Proposition~\ref{prop:8.5} by the techniques of~\cite{sayas_new_analysis}. The estimates read:
  \begin{align*}
    \|{u^{(\ell)}(t)}\|_{L^2(\R^d \setminus \overline{\Omega})}&\leq C t \sum_{j=\ell}^{\ell+1}\max_{\tau \leq t} \|{g^{(j)}}\|_{H^{1/2}(\Gamma)}.
  \end{align*}
  
  For simplicity of notation, we only consider the case $q<p-1$. All the other cases follow analogously but giving different rates when applying
  the abstract theory.  
  By Proposition~\ref{prop:interpolation_space_heat}, we can apply Propositions~\ref{prop:AMP1} or~\ref{prop:AMP2} with $\mu=1/4$, depending on whether
  we are in the setting $\alpha=0$ or $\alpha=1$. This gives estimate~\eqref{eq:10.3}.

  Applying Theorem~\ref{theorem:3.3}, we get the following
  convergence in the graph norm of $\AA_{\star}$:
  \begin{align}
    \label{eq:10.6}
    \|{\laplace u^{k}(t_n) - \laplace u(t_n)}\|_{L^2(\R^d \setminus \overline{\Omega})}&\leq C
                                                                                              (1+T^2)  k^{q+\alpha-1+1/4}\sum_{\ell=q}^{p+3}{\|{g^{(\ell)}}\|_{H^{1/2}(\Gamma)}}.
  \end{align}
  Since $\gamma^+ u^{k}(t_n) = \gamma^+ u(t_n)$, integration by parts and the Cauchy-Schwarz inequality give:
  \begin{align*}
    \|{\nabla u^{k}(t_n) - \nabla u(t_n)}\|_{L^2(\R^d \setminus \overline{\Omega})}^2
    &=-\big( \laplace u^{k}(t_n) - \laplace u(t_n), u^{k}(t_n) -  u(t_n)\big)_{L^2(\R^d \setminus \overline{\Omega})} \\
      &\leq \|{ \laplace u^{k}(t_n) - \laplace u(t_n)}\|_{L^2(\R^d \setminus \overline{\Omega})}\|{u^{k}(t_n) -  u(t_n)}\|_{L^2(\R^d \setminus \overline{\Omega})}.
  \end{align*}
  Estimate~\eqref{eq:10.4} then follows from \eqref{eq:10.3} and~\eqref{eq:10.6}.
  For the estimate~\eqref{eq:10.5} of the density, we fix $\xi \in H^{1/2}(\Gamma)$, and let $v$ denote a lifting to $H^1(\R^d)$.
  We calculate
  \begin{align*}
    \big<{\lambda -\lambda^k,\xi}\big>_{\Gamma}
    &=\big(-\laplace u + \laplace u^k,v\big)_{L^2(\Omega)} + \big(\nabla u - \nabla u^k,\nabla v\big)_{L^2(\R^d \setminus \overline{\Omega})} \\
    &\leq
      \begin{multlined}[t][11cm]
        \big( k^{1/2}\,\|{\laplace u - \laplace u^k}\|_{L^2(\Omega)} + \|{\nabla u - \nabla u^k}\|_{L^2(\R^d \setminus \overline{\Omega})}\big) \\ %\vspace{3mm}       
        \times \big(k^{-1/2} \|{v}\|_{L^2(\Omega)}  + \|{\nabla v}\|_{L^2(\R^d \setminus \overline{\Omega})}\big).
        \end{multlined}
  \end{align*}
  We are still free to pick the precise lifting $v$. Doing so as in \cite[Proposition 2.5.1]{sayas_book}, we get
  \[
    \inf\{ k^{-1/2}\|v\|_{L^2(\R^d)}+\|\nabla v\|_{L^2(\R^d)} :  v\in H^1(\R^d),\,\gamma v=\xi\}
    \lesssim \max\{1,k^{-1/4}\}\| \xi\|_{H^{1/2}(\Gamma)}.
  \]

  The result then follows from the previous estimates.
\end{proof}

\subsubsection{Numerical example}
In order to demonstrate that the estimate~\eqref{eq:10.5} is sharp, we consider a simple model problem.
Following~\cite{sauter_veit}, we take $\Omega$ to be the unit sphere and
consider a right hand side $g(x,t)$ of the form
$$
g(x,t):=\psi(t)Y^m_n(x),
$$
where $Y^m_n$ is the spherical harmonic of degree $n$ and order $m$. It is well-known
that the spherical harmonics are eigenfunctions of the pertinent boundary integral operators.
Most notably for us, we have
\begin{align*}
V(s) Y^m_n &= \mu_n(s)  Y^m_n \quad \text{with } \quad  \mu_n:=-s \,j_n(\ii \,s) \,h^{(1)}_n(\ii \,s),
\end{align*}
where $j_n$ denotes the spherical Bessel functions and $h_n^{(1)}$ is the spherical Hankel function of the first kind.
Due to this relation, solving \eqref{eq:heat_eq_bem} becomes a purely one dimensional problem,
i.e., we can write $\lambda(x,t)=\widetilde{\lambda}(t) Y^m_n(x)$ and
the solution can be easily computed to very high accuracy. For our experiments we chose $n=2$.

We compare the 3-stage and 5-stage Radau~IIA methods (see~\cite{hairer_wanner_2} for their definitions).
These methods have stage order $3$ and $5$ respectively and
both are stiffly accurate and satisfy Assumption~\ref{ass:1.3}. We therefore expect convergence rates
for the density $\lambda$ of order $3.5$ and  $5.5$.
Since the exact solution is not available, we compute the difference to an approximation with step-size $k/4$ and
use this as an approximation to the discretization error. The results can be seen in Figure~\ref{fig:convergence_heat}.
We observe that the results are in good agreement with our predictions.

\begin{figure}[htb]
  \centering
  \includeTikzOrEps{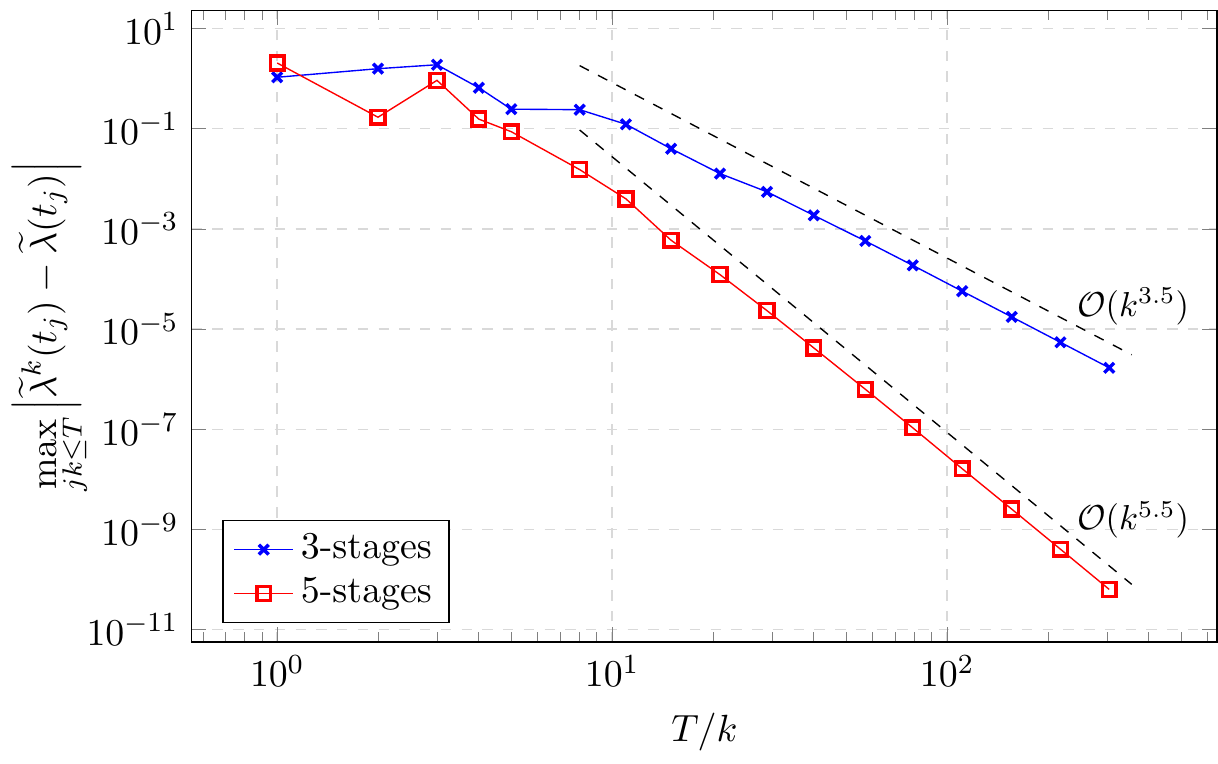}
  \caption{Convergence for the density $\widetilde{\lambda}$ for the heat
    conduction problem (cf. Section~\ref{sect:heat}), comparing Radau~IIA methods.}
  \label{fig:convergence_heat}
\end{figure}

\appendix
\section{Interpolation of Sobolev spaces}
\label{appendix:interpolation_of_sobolev}
In this  appendix we prove that in Lipschitz domains and for certain parameters $\mu$,
the spaces $[L^2(\Omega),H^1_0(\Omega)]_{\mu}$
contain functions with non-vanishing boundary conditions.
Such estimates are the main ingredient when determining the convergence rate of Runge-Kutta methods
using the theory developed in the previous sections. For $\mu<1/2$, it is well known
that the fractional Sobolev spaces $H^\mu(\Omega)=[L^2(\Omega),H^1(\Omega)]_{\mu,2}$ and $\widetilde{H}^{\mu}(\Omega)=[L^2(\Omega),H^1_0(\Omega)]_{\mu,2}$
coincide (see e.g. \cite[Theorem 3.40]{mclean} together with the results in \cite[Appendix B]{mclean}
to identify the Sobolev spaces with the interpolation space).
We prove that when interpolating using the index $\infty$, the critical value $\mu=1/2$ is also
admissible, provided that some further regularity is provided.

In order to state our result, we need additional notation, notably
we define interpolation spaces for $q\in [1,\infty)$ as 
\begin{align}
  \label{eq:definition_interpolation_norm_besov}
  \norm{u}_{[\mathcal{X}_0,\mathcal{X}_1]_{\mu,q}}^q
  &:=\int_{0}^{\infty}{t^{-\mu} \left[\inf_{v \in \mathcal{X}_1} \norm{u-v}_{\mathcal{X}_0} + t \norm{v}_{\mathcal{X}_1}\right]^q \frac{dt}{t}},
\end{align}
and introduce the following Besov spaces:
\begin{align}
  \label{def:besov}
  B^\mu_{2,q}(\Omega):=\left[L^2(\Omega), H^1(\Omega)\right]_{\mu,q}
  \qquad \text{and}\qquad
  \widetilde{B}^\mu_{2,q}(\Omega):=\left[L^2(\Omega), H^1_0(\Omega)\right]_{\mu,q}.
\end{align}
For $t >0$, we define the strip
\begin{align}
  \Omega_t:=\big\{x \in \Omega: \operatorname{dist}(x,\partial \Omega) < t\big\},
\end{align}
which will play an important role in the following proofs.

\begin{theorem}
  \label{thm:interpolation_spaces_coincide}
  Let $\Omega$ be either a bounded Lipschitz domain or the complement of a bounded Lipschitz domain.
  Fix $\mu \in (0,1/2]$. Then
  \begin{align*}
    B^\mu_{2,1}(\Omega)\subseteq  \widetilde{B}^\mu_{2,\infty}(\Omega)
  \end{align*}
  with equivalent norms. The implied constant depends on $\Omega$ and $\mu$.
\end{theorem}
\begin{proof}
  For simplicity, assume that $\Omega$ is bounded. We focus on the
  case $\mu=1/2$, the general one follows by an interpolation argument.  
  Consider $u \in B^\mu_{2,1}(\Omega)$.
  For fixed $t>0$, we select $v(t) \in H^1(\Omega)$ as function almost realizing the infimum appearing in the interpolation norm, i.e.,
  \begin{align*}
    \norm{u-v(t)}_{L^2(\Omega)} + t \norm{v(t)}_{H^1(\Omega)} \leq 2 \inf_{w \in H^1(\Omega)} \left(\norm{u-w}_{L^2(\Omega)} + t \norm{w}_{H^1(\Omega)}\right).
  \end{align*}
  By~\cite[Lemma]{bramble_scott}, the following estimate holds for all $t\geq 0$:
  \begin{align*}
    \norm{v(t)}_{B^\mu_{2,1}(\Omega)}\leq 3 \norm{u}_{B^\mu_{2,1}(\Omega)}.
  \end{align*}

  We consider a smooth cutoff function $\chi_t: \Omega \to [0,1]$ satisfying:
  \begin{align}
    \label{eq:properties_of_cutoff}
    \chi_t(x) &\equiv 0 \; \text{ on $\Omega_t$},  \qquad
    \chi_t(x) \equiv 1 \; \text{ on $\Omega \setminus \Omega_{2t}$} \qquad \text{ and } \quad \norm{\nabla \chi_t}_{L^{\infty}} \lesssim t^{-1}.
  \end{align}  
  
  We then define $\widetilde{v}(t):=\chi_t v(t) \in H_0^1(\Omega)$ and calculate:
  \begin{align*}
    \norm{u-\widetilde{v}(t)}_{L^2(\Omega)}
    &\leq \norm{u-{v}(t)}_{L^2(\Omega)} + \norm{(1-\chi_t)v}_{L^2(\Omega_{2t})}\\
    &\lesssim \norm{u-{v}(t)}_{L^2(\Omega)} + t^{1/2}\norm{v}_{B^{1/2}_{2,1}(\Omega)}
  \end{align*}
  where we used the fact that $1-\chi_t$ vanishes on $\Omega \setminus \Omega_{2t}$
  and
  applied~\cite[Lemma 2.1]{LMWZ10} to estimate the $L^2$-norm there.

  Similarly,
  \begin{align*}
    t\,\norm{\widetilde{v}(t)}_{H^1{\Omega}}
    &\lesssim t\, \norm{v}_{H^1(\Omega)} + t\, \norm{(\nabla\chi_t) v}_{L^2(\Omega)}  
    \lesssim t\, \norm{v}_{H^1(\Omega)} +  \norm{v}_{L^2(\Omega_{2t})} \\
    &\lesssim t \norm{v}_{H^1(\Omega)} + t^{1/2} \norm{v}_{B^{1/2}_{2,1}(\Omega)}.
  \end{align*}

  For the interpolation norm, we therefore get:
  \begin{align*}
    \norm{u}_{\widetilde{B}^{1/2}_{2,\infty}(\Omega)} 
    &\lesssim \operatorname{ess\,sup}_{t>0}\!\Big[ t^{-1/2}
      \big(\!\norm{u-v(t)}_{L^2(\Omega)} \!+\! t \norm{v(t)}_{H^1(\Omega)}
      + t^{1/2} \norm{v}_{B^{1/2}_{2,1}(\Omega)}
      \big)\Big] \\
    &\lesssim \norm{v}_{B^{1/2}_{2,1}(\Omega)}
      +\norm{u}_{\left[L^2(\Omega), H^1(\Omega)\right]_{\mu,\infty}}
      \lesssim \norm{u}_{B^{1/2}_{2,1}(\Omega)}
      +\norm{u}_{B^{1/2}_{2,\infty}(\Omega)} \\
    &\lesssim \norm{u}_{B^{1/2}_{2,1}(\Omega)}.
  \end{align*}

  If $\Omega$ is the exterior of a bounded Lipschitz domain, the proof applies almost verbatim as all important steps can be localized to a neighborhood of the boundary.
\end{proof}
\begin{remark}
  The use of the second parameter $\infty$ in the interpolation norm is crucial for Theorem~\ref{thm:interpolation_spaces_coincide}
  to hold in the case $\mu=1/2$. For $L^2$-based interpolation it is well known that the interpolation space
  $\left[L^2(\Omega), H^1_0(\Omega)\right]_{1/2,2}$ is the Lions-Magenes space $H^{1/2}_{00}(\Omega)$, see \cite[Chapter 33]{tartar07},
  which is distinct from $H^{1/2}(\Omega)$.
\end{remark}

When considering the Neumann problem in Section~\ref{sec:9},
we need to create a lifting to a vector field with a given normal jump in $L^2$.
In general, such liftings do not have $B^{1/2}_{2,1}$-regularity. Thus Theorem~\ref{thm:interpolation_spaces_coincide} is not applicable. Instead, we have a modified construction.

\begin{lemma}
  \label{lemma:integral_strip_using_mf}
  Let $\Omega$ be a bounded Lipschitz domain or the exterior of a bounded Lipschitz domain
  with boundary $\Gamma:=\partial \Omega$.
  For $C > 0$, $c>0$ fixed with $c$ sufficiently small, define the non-tangential maximal function
  \begin{align*}
    N(\nabla u)(x)
    &:=\!\!\sup_{y \in \Theta(x)}{\abs{\nabla u(y)}}, \; \text{where } \;
      \Theta(x):=\{y \in \Omega: \abs{x-y}\leq \max(c,C\operatorname{dist}(y,\Gamma))\}.
  \end{align*}
  Let $u \in H^1(\Omega)$ be harmonic and satisfy $N(\nabla u) \in L^2(\Gamma)$.

  Then for $t>0$ we can bound the $L^2$ norm on strips $\Omega_t$ by
  \begin{align}
    \norm{u}_{L^2(\Omega_t)}
    &\lesssim t^{1/2} \norm{N(\nabla u)}_{L^2(\Gamma)}.
  \end{align}
\end{lemma}
\begin{proof}
  We focus on a single chart in the parametrization of (a vicinity of ) $\Gamma$.
  Let $\mathcal{O} \subseteq \Omega$ and $\mathcal{D} \subseteq \R^{d-1}$ be open,
  $\bs r \in \R^n$,  $\varphi: \mathcal{D}\to \R$,
  $y_0: \mathcal{D} \to \R$
  such that we can write
  $$
  \Omega_t \cap \mathcal{O}
  = \big\{ (x, \varphi(x)+ y \bs r): x \in \mathcal{D}, \text{ and }y \in (0,y_0(x)) \big\}.
  $$
  By the Lipschitz assumption, we note that $y_0(x) \lesssim C t$.
  Following the considerations in \cite[Appendix A.4]{CGLS12}, one can see
  that as long as $C$ in the definition of $\Theta$ is taken sufficiently large,
  we have that for all $x \in \mathcal{D}$
  $$
  \{(x, \varphi(x)+ \tau \bs r): y \in (0,y_0(x)) \}
  \subseteq \Theta(x, \varphi(x)).
  $$
 
  We calculate
  \begin{align}
    \norm{u}_{L^2(\Omega_t \cap \mathcal{O})}^2
    &=\int_{x \in \mathcal{D}}{
      \int_{y=0}^{y_0(x)}{\abs{\nabla u(x,\varphi(x)+y \bs r)}^2 \,dy}\,dx} \\
    &\lesssim \int_{x \in \mathcal{D}}{\int_{y=0}^{y_0(x)}
      {\big(N(\nabla u)(x)\big)^2 \,dy}\,dx} \\
    &\leq t\int_{x \in \mathcal{D}}{{\big(N(\nabla u)(x)\big)^2 \,dx} } 
      \leq  t \norm{N(\nabla u)}_{L^2(\Gamma)}^2.
  \end{align}
  Repeating the same calculation for all boxes needed to parametrize
  a neighborhood of $\Gamma$ then concludes the proof.
\end{proof}

\begin{theorem}
  \label{thm:lifting_of_neumann}
  Let $\Omega$ be a bounded Lipschitz domain or the exterior of a bounded Lipschitz domain
  and write
  $\mathbf H_0(\mathrm{div},\Omega):=\{ \bs w \in \mathbf H(\mathrm{div},\Omega): \gamma^-_{\nu} \bs w =0\}$.
  
  For every $g \in L^2(\Omega)$, there exists a function $\bs w \in \mathbf H(\mathrm{div},\Omega)$
  such that
  \begin{align}
    \label{eq:lifting_of_neumann_est}
    \gamma_{\nu} \bs w &=g \qquad \text{and} \qquad
    \norm{\bs w}_{
    [L^2(\Omega), \mathbf H_0(\mathrm{div},\Omega)]_{1/2,\infty}}
    \lesssim \norm{g}_{L^2(\Omega)}
  \end{align}
  with a constant depending only on $\Omega$.

\end{theorem}
\begin{proof}
  For simplicity, assume that $\Omega$ is bounded. By performing an appropriate
  cutoff away from $\partial \Omega$, all arguments can be localized.
  First, consider the case $\int_{\Gamma} g =0$. Let $u$ solve the Neumann problem
  $$
  \laplace u =0 \text{ in $\Omega$ } \qquad \partial_n u = g \text{ on $\partial \Omega$}.
  $$

  In addition to $u \in H^1(\Omega)$, by~\cite{jerison_kenig}(see also
  \cite[Theorem A.6]{CGLS12}),
  such harmonic functions $u$ also satisfy
  $$\norm{N(\nabla u)}_{L^2(\Gamma)} \leq \norm{g}_{L^2(\Gamma)}.$$

  For fixed $t>0$ we again pick a smooth cutoff function $\chi_t$ satisfying~(\ref{eq:properties_of_cutoff}).
  We set $\bs w:=\nabla u$ and calculate
  using Lemma~\ref{lemma:integral_strip_using_mf}:
  \begin{align*}
    \norm{\bs w}_{
    [L^2(\Omega), \mathbf H_0(\mathrm{div},\Omega)]_{1/2,\infty}}
    &\lesssim
      \operatorname{esssup}_{t \geq 0}\Big(
      t^{-1/2} \big[ \norm{(1-\chi_t) \bs w }_{L^2(\Omega)} 
      + t \norm{\chi_t \bs w }_{\mathbf H(\mathrm{div},\Omega)} \big]\Big) \\
    &\lesssim
      \operatorname{esssup}_{t \geq 0}\Big(
      \norm{N(\nabla u)}_{L^2(\Gamma)}
      + t^{1/2}\norm{ \nabla u}_{L^2(\Gamma)}
      \Big) \\
      &\leq \norm{g}_{L^2(\Gamma)}.
  \end{align*}
  In the case $\int_{\Gamma} g \neq 0$, we
  construct $u$ as before using the Neumann data $g-\int_{\Gamma} g$
  and then set
  $$
  \bs w:=\nabla u + \int_{\Gamma} g.
  $$
  It is easy to see using a similar cutoff technique, that
  the interpolation norm in~(\ref{eq:lifting_of_neumann_est})
  is bounded for constant functions.  
\end{proof}

%\clearpage
\textbf{Acknowledgments:} Work of A.R. was financially supported by the Austrian Science Fund (FWF) through the 
doctoral school ``Dissipation and Dispersion in Nonlinear PDEs'' (project W1245)
and the research program ``Taming complexity in partial differential systems'' (grant SFB F65).
FJS is partially supported by NSF-DMS grant 1818867. Part of this work was developed while FJS was a Visiting Professor at TU Wien.

\bibliographystyle{halpha}
\bibliography{literature}
\end{document}